\documentclass[11pt, a4paper]{article}
\usepackage[english]{babel}
\usepackage{amsmath,amssymb,xcolor,amsthm,a4wide}
\usepackage{authblk}

\usepackage{comment}
\usepackage{graphicx}
\usepackage{appendix}
\usepackage[hidelinks]{hyperref}

\usepackage{mathtools}

\newcommand{\cdummy}{\cdot}
\newcommand{\mathd}{\mathrm{d}}

\newcommand{\tmem}[1]{{\em #1\/}}
\newcommand{\tmop}[1]{\ensuremath{\operatorname{#1}}}

\newcommand{\jacdat}{\sqrt{|G|}}

\newcommand{\half}{\frac{1}{2}}

\newcommand{\keywords}[1]{%
  \par\smallskip
  \noindent\textbf{Keywords: }#1
}

\newcommand{\MSC}[2][]{%
  \par\smallskip
  \noindent\textbf{MSC#1: }#2
}

\newtheorem{theorem}{Theorem}

\newtheorem{definition}[theorem]{Definition}
\newtheorem{remark}[theorem]{Remark}
\newtheorem{lemma}[theorem]{Lemma}

\title{ Surface Dean--Kawasaki equations}

\author[1]{John Bell}
\author[2]{Ana Djurdjevac}
\author[3]{Nicolas Perkowski}

\affil[1]{Lawrence Berkeley National Laboratory, Berkeley, California, 94720, USA}
\affil[2]{University of Oxford, Mathematical Institute, Oxford, UK, 

and Freie Universität Berlin, Arnimallee 6, 14195 Berlin, Germany}

\affil[3]{Freie Universität Berlin, Arnimallee 7, 14195 Berlin, Germany, 

and Max-Planck Institute for Mathematics in the Sciences, Leipzig, Germany}

\begin{document}
\maketitle
\begin{abstract}
We consider stochastic particle dynamics on hypersurfaces represented in Monge gauge parametrization. Starting from the underlying Langevin system, we derive the surface Dean–Kawasaki (DK) equation and formulate it in the martingale sense. The resulting SPDE explicitly reflects the geometry of the hypersurface through the induced metric and its differential operators. Our framework accommodates both pairwise interactions and environmental potentials, and we extend the analysis to evolving hypersurfaces driven by an SDE that interacts with the particles, yielding the corresponding surface DK equation for the coupled surface–particle system. We establish a weak uniqueness result in the non-interacting case, and we develop a finite-volume discretization preserving the fluctuation–dissipation relation. Numerical experiments illustrate equilibrium properties and dynamical behavior influenced by surface geometry and external potentials.


\MSC[MSC Classification]{ 60H15, 60J60, 60H35, 65C30, 60K35}



\end{abstract}

\section{Introduction}

Transport processes constrained to curved surfaces arise in a wide range of physical, biological, and chemical systems. Whenever the motion of microscopic particles: proteins, agents, etc. -- is restricted to a lower-dimensional manifold, the interplay between geometry, stochasticity, and interactions becomes an important feature of the resulting dynamics. Such scenarios appear for example in statistical physics (active or passive particles constrained to manifolds), agent based modeling in social dynamics  (target tracking system on the planet \cite{choi2023multi}), molecular dynamics \cite{ciccotti2008projection,LelievreRoussetStoltz,lindahl2008membrane}, or consensus-based optimization \cite{Fornasier2021}. 



Cell biology, in particular, offers a representative setting for such systems \cite{bressloff2013stochastic, lipowsky1995structure}.
The cell membrane is a  (thermally fluctuating) surface composed of lipids and embedded proteins, and it hosts numerous processes crucial for cellular function. A prototypical membrane model is the thermally excited Helfrich elastic membrane \cite{naji2007diffusion,reister2005lateral}.  Lateral diffusion of membrane-bound molecules is one of the primary mechanisms governing signaling, transport, and regulation at the cellular interface (see, e.g., \cite{bressloff2013stochastic}). The study of such processes has attracted sustained attention over the past decades \cite{castro2010brownian,lipowsky1995structure,saxton1997single}. While some works investigate particle diffusion on static membranes \cite{reister2005lateral}, 
others concentrate on membranes with thermal fluctuations that lead to the change of the shape of a membrane over time \cite{gov2004membrane,reister2007hybrid,  reister2010diffusing}.
Besides the interaction between particles, there is also induced interaction between the particles and the shape of the membrane.  


From a mathematical perspective, these systems provide prototypical examples where geometry and noise couple in a nontrivial way, requiring new analytical and numerical tools.
The particles are often modeled as interacting Brownian diffusions on manifolds. Brownian motion on a hypersurface can be defined using the Laplace-Beltrami generator \cite{Hsu}, but can be also represented via Langevin equations. As highlighted in \cite{naji2007diffusion}, an important advantage of the Langevin description is its flexibility to incorporate extensions that account for both particle–particle and membrane-particle interactions (for instance, active protein inclusions \cite{lin2006nonequilibrium}).

We consider the case when the surface is described in the Monge gauge,  meaning that it is represented as the graph of a given height function, as for example in \cite{DuncanElliottPavliotisStuart}. The advantage of this representation is that all differential operators can be easily represented explicitly in global coordinate space using appropriate Euclidean differential operators.

The simulations of such  particle-membrane systems are often very costly when the system becomes large, which is typically the case in the applications mentioned above. The main goal of this work is to propose a reduced model for particle systems on (evolving) curved domains with pairwise interactions and particle-membrane coupling, while still capturing finite-size effects.
A well-established framework for such reductions is the description of the dynamics of the empirical measure through  Dean–Kawasaki equation \cite{dean1996langevin,kawasaki1994stochastic}. The study of these equations, from both computational \cite{donev2010accuracy,helfmann2021interacting} and analytical \cite{cornalba2023dean, Cornalba2023Density, djurdjevac2024weak,djurdjevac2025weak, Fehrman2023, fehrman2024well,konarovskyi2020dean,Konarovskyi} perspectives, has evolved rapidly over the last decade. Recent work by Dello Schiavo \cite{DelloSchiavo2024massive} derives the Dean–Kawasaki equation for independent particles associated with a reversible Dirichlet form, using the framework of Hilbert tangent bundles. This covers, in particular, general Riemannian manifolds.  Nevertheless, to the best of our knowledge, no counterpart of the Dean–Kawasaki equation has been derived for interacting particle systems on curved, evolving domains, and there has been no numerical framework for this geometric setting. Closing this gap  is the main contribution of the present work.

Starting from the Langevin dynamics describing both pairwise particle interactions and particle-membrane coupling, we apply Itô’s formula to derive the evolution equation for the empirical distribution. As in the flat case, a key difficulty in this derivation is that the resulting noise term is not closed, making a direct formulation in terms of the empirical density problematic. For this reason, a more appropriate viewpoint is provided by the martingale formulation, which allows us to characterize the quadratic variation of the noise. The martingale formulation suggests a formal nonlinear transport-type equation for the (formal) density of the empirical distribution, which serves as basis for numerical discretizations. A distinct feature arising on hypersurfaces concerns the choice of the correct noise representation: in this geometric setting, one may consider the (formal) density of the empirical distribution either with respect to Lebesgue measure on the coordinate space or with respect to the intrinsic surface area measure. We analyze both possibilities, which lead to two different formulations of the resulting equation. To justify the chosen noise structure, we further establish that the fluctuation-dissipation relation is satisfied. In addition, we prove weak uniqueness for the surface Dean–Kawasaki equation describing non-interacting particles on a fluctuating hypersurface.


For the numerical approximation, we exploit the conservative structure of the equation and employ a finite-volume discretization in space combined with an Euler--Maruyama scheme in time. A major challenge arises from the fact that the matrices appearing in the divergence formulation are generally full, leading to coupling through off-diagonal terms. To address this issue, we adopt one-sided approximations of the surface gradient operators and construct a discrete metric divergence operator that is adjoint to the discrete gradient with respect to the surface area measure. We show that with this type of the discretization, the fluctuation dissipation relation is preserved. We present two main classes of numerical experiments. The first focuses on equilibrium statistics of the linearized surface Dean--Kawasaki equation, while the second examines how the interplay between particle dynamics and geometry influences the spreading of particles. The results indicate that the surface geometry plays a central role in governing how particles move and how they ultimately spread across the domain. Additional numerical experiments illustrate how an external potential impacts the particle dynamics.

The paper is organized as follows. Section \ref{sec:preliminaries} reviews the geometric preliminaries, introducing the Monge gauge parametrization and the differential operators required for the formulation of the surface equation. Section \ref{sec:derivation} presents the derivation of the surface Dean–Kawasaki equation starting from the underlying Langevin dynamics. To highlight the main ideas, we first examine the case of independent Brownian particles on a hypersurface and then derive the full surface Dean–Kawasaki equation for interacting particles coupled to an evolving hypersurface whose motion is governed by a stochastic differential equation. In Section \ref{sec:discretization}, we present a finite volume discretization of the two-dimensional surface Dean--Kawasaki equation, including the effect of a possible external potential. Section \ref{sec:experiments} provides computational examples that validate the proposed numerical scheme in the absence of an external potential and demonstrate its performance when an external potential is present.

\section{Preliminaries}\label{sec:preliminaries}

\subsection{Geometric setting }

We will consider surface that can be presented in Monge gauge (graph) parametrization. We consider the Monge-gauge representation mainly for concreteness, as it allows an explicit local-chart formulation of the Langevin dynamics and the associated Dean–Kawasaki equation, thereby making both the derivation and numerical implementation more convenient. An extension of this framework to general Riemannian manifolds is a natural next step and will be considered in future work. Let\footnote{The analysis would extend without difficulty to $\mathcal{D}=\mathbb R^d$ under suitable assumptions on $H$, or to domains $\mathcal{D} \subset \mathbb{R}^d$ with (piecewise) smooth boundary under homogeneous Neumann boundary conditions.}  $\mathcal{D}=\prod_{i=1}^d(\mathbb R/L_i\mathbb Z)$, for $L_1,\dots, L_d > 0$. For a $C^3$ height (Monge gauge) function $H: \mathcal{D}  \to \mathbb{R}$, the surface $\Gamma$ is then parametrized over $\mathcal{D}$ in the following way
\begin{equation*}
    \Gamma = \{ (x, H(x)) : x \in \mathcal{D}\}.
\end{equation*}
Hence, $\Gamma$ is the graph of the function $H$ and is parametrized by the Monge gauge map $\Phi:\mathcal{D} \to \mathbb{R}^{d+1}$ with  $\Phi(x):=(x,H(x))$, $x\in \mathcal{D}$, and the tangent basis vectors are given by
\begin{equation*}
    \tau_i = \partial_i \Phi = e_i + (\partial_i H) e_{d + 1}, \qquad i = 1,
   \ldots, d,
\end{equation*}
where $e_i$ denotes the $i$-th canonical basis vector of $\mathbb{R}^d$. Then the induced metric is given by
\begin{equation*}
     G (x)  = [\tau_i \cdot \tau_j]_{i,j=1}^d=D\Phi(x)^T D\Phi(x)=
I_d + \nabla H (x) \otimes \nabla H (x) \in
   \mathbb{R}^{d \times d},  \quad x \in \mathcal{D} \subset \mathbb{R}^d, 
\end{equation*}   
where $D \Phi$ denotes the Jacobian matrix of the parametrization $\Phi$ and  $\operatorname{I}_d$ is the identity matrix in $\mathbb{R}^{d\times d}$. 
We denote the determinant of the matrix $G$ by 
\begin{equation*}
    | G | (x) := \det (G(x))=
    1 + |\nabla H(x)|^2 ,
\end{equation*}
where $| \cdot |$ denotes the Euclidean distance. 
It follows directly that for a unit vector $e \in \mathbb{R}^d$ 
\[
1 \leq e \cdot G(x)e \leq |G|(x) \quad \forall x \in \mathcal{D}, 
\]
hence $G$ is uniformly elliptic and $C^2$. The inverse matrix $G^{-1}$ has components that we denote by $G^{ij}$ and is
given by
\[ G^{- 1} = \operatorname{I}_d - \frac{\nabla H \otimes \nabla H}{1 + | \nabla H |^2}.  \]
 
The unit outward normal is given by
\begin{equation*}
    n = \frac{1}{\sqrt{ | G |}} 
\begin{pmatrix} 
     - \nabla H\\
     1
\end{pmatrix}.
\end{equation*}
The surface area element is defined as 
\begin{equation}\label{def:nu}
    \nu(dx):=\sqrt{|G|(x)} \mathd x 
\end{equation}
and   for a scalar field $\tilde{f} : \Gamma \rightarrow \mathbb{R}$, after
using the change of the coordinates $y = \Phi (x),$ the surface measure $S$ on $\Gamma$ is given by
\begin{equation}\label{partial_int}
  \int_{\Gamma} \tilde{f} (y) {\mathd S} (y) = \int_\mathcal{D} \tilde{f} (\Phi (x)) \sqrt{ | G | (x)}
  {\mathd x}.
\end{equation}

Next, we introduce differential operators. To define the tangential gradient, we introduce the projection onto the tangent space, denoted by
$\Pi:\mathbb{R}^{d+1} \to \mathbb{R}^d$, which is defined
as
\[ \Pi := \operatorname{I}_{d + 1} - n \otimes n .\]
For the smooth function $\tilde{f}:\Gamma \to \mathbb{R}$ defined on the surface, the tangential gradient is defined as
\begin{equation}\label{eq:tangential-gradient}
    \nabla_\Gamma \tilde{f}(y) := \Pi(y) \nabla \tilde{f}(y), \quad y \in \Gamma,
\end{equation}
where the ambient gradient is applied to any extension of the function $\tilde{f}$
to the neighborhood of the surface, but since it can be shown that the value does not depend on this extension we use the same notation $\tilde{f}$. In the Monge gauge representation of $\Gamma$, every function $\tilde{f}:\Gamma \to \mathbb{R}$  on the surface can be identified with a function on the coordinate domain $\mathcal{D}$:
\begin{equation*}
    {f}:\mathcal{D} \to \mathbb{R}, \quad {f}(x) := \tilde{f}(\Phi(x)) . 
\end{equation*}
Using the chain rule, we can express the tangent vector $\nabla_\Gamma$ in local coordinates in the following way:
\begin{equation*}
    \nabla_\Gamma \tilde{f}(\Phi(x)) = D\Phi(x) G^{-1}(x) \nabla{f}(x)= \begin{pmatrix}
        G^{-1} \nabla {f}(x) \\
        \nabla H (x)   \cdot G^{-1}(x) \nabla {f} (x)
         \end{pmatrix}, \, \qquad x \in \mathcal{D} .
\end{equation*}
Since we will consider everything in the coordinate domain, we define the metric gradient of the function $f:\mathcal{D}\to \mathbb{R}$ as the $d$-dimensional upper block of the tangential gradient:
\begin{equation}\label{nablaG}
\nabla_G f := G^{-1}\nabla f .
\end{equation}
Furthermore, we define the metric divergence, that is the adjoint of the metric gradient, in the following way
\begin{equation*}
    \int_\mathcal{D} (\text{div}_G\cdot w) \varphi \sqrt{ | G |} {\mathd x} := -  \int_\mathcal{D} (w,\nabla_G \varphi)_G \sqrt{ | G |} {\mathd x} \,= - \int_\mathcal{D} w
   \cdot G \nabla_G \varphi \sqrt{ | G |} {\mathd x}, \quad \varphi \in C^{\infty}_c
   (\mathcal{D}),
\end{equation*}
where
\[
    (a,b)_G := a \cdot Gb.
\]
Using the definition of the tangential  gradient and chain rule, the metric divergence can be written using the standard Euclidean divergence as
\begin{equation}\label{def:divG}
     \text{div}_G f = \frac{1}{\sqrt{| G |}} \nabla \cdummy \left(
   \sqrt{| G |} f \right) .
   \end{equation}
We define the Laplace-Beltrami operator in the natural way
\begin{equation}\label{def:LB}
    \Delta_G f:= \text{div}_G \nabla_G f= \frac{1}{\sqrt{| G |}} \nabla
   \cdummy \left( \sqrt{| G |} G^{- 1} \nabla f\right).
\end{equation}
Note that $\Delta_G$ is symmetric in $L^2 \left( \nu \right)$.

\subsection{Brownian motion on the hypersurface}

We recall the definition of a Brownian motion on a hypersurface \cite{DuncanElliottPavliotisStuart}, which is a special case of a more general definition given in \cite[Proposition 3.2.1]{Hsu}.



For convenience,  from now on we periodically extend functions on $\mathcal{D}$ to $\mathbb R^d$. For example, this allows us to directly describe the Brownian motion on the hypersurface by an SDE, instead of having to consider  solutions modulo $\prod_{i=1}^d L_i \mathbb Z$.

\begin{definition}
    Let $(\Omega, \mathcal{F},\mathbb{P})$ be a complete probability space endowed with a right-continuous filtration $(\mathcal{F}_t)_{t \geq 0}$. An $\mathbb{R}^d$-valued process $X$ defined on $\Omega \times [0,T]$ is called a Brownian motion on $\Gamma$ started at $X_0=x \in \mathbb{R}^d$ if $X$ is adapted, almost surely continuous, and if for every $f \in C^2(\mathbb R^d)$,
    \[
    f(X_t) - f(x) -\int_0^t \Delta_G f(X_s)\mathd s,\qquad t\ge 0,
    \]
    is a local martingale starting from $0$, where $\Delta_G$ is the Laplace-Beltrami operator in local coordinates on $\mathbb{R}^d$, given by \eqref{def:LB}.
\end{definition}

The process $X$ has the following representation through an SDE:

\begin{equation}\label{eq:BMonSurface}
    \mathd X_t = b(X_t) \mathd t + \sqrt{2}
   G^{- 1 / 2} (X_t) \mathd B_t,
   \end{equation}
where 
\begin{equation}\label{b}
    b(X_t):=\frac{1}{\sqrt{| G |}} \nabla \cdummy \left(
   \sqrt{| G |} G^{- 1} \right)
\end{equation}
and $B$ is a $d$-dimensional standard Brownian motion.

Note that the coefficients of the SDE are Lipschitz continuous, because $H 
\in C^3$ is periodic and thus its derivatives are bounded. For $H \in C^2$ we would get a singularity in the drift coefficient, and it is well understood that the noise would regularize this singularity \cite{Veretennikov1981}; for simplicity we stick with the Lipschitz setting.


\begin{remark} \label{rem:StrBM}
    Note that there is another way of defining the Brownian motion on $\Gamma$ using the projection on the tangent space and the ambient Brownian motion $B^{d+1}$ from $\mathbb{R}^{d+1}$ \cite[Example 8.4]{EVanden-Eijnden}. More precisely, the process defined as $\mathd Y_t := \Pi \circ \mathd B^{d+1}_t$, where $\circ$ denotes the Stratonovich integral, and whose infinitesimal generator is given by the Laplace--Beltrami operator.
\end{remark}

\section{Derivation of surface Dean--Kawasaki equation}\label{sec:derivation}

\subsection{Surface Dean--Kawasaki equation on a fixed surface: non-interacting, potential-free case}

Let now $X^1, \ldots, X^N$ be i.i.d. copies of the Brownian motion $X$ given by \eqref{eq:BMonSurface}, potentially with different
initial conditions. Furthermore, let
\begin{equation}
        \mu_t = \frac{1}{N} \sum_{i = 1}^N \delta_{X^i_t}
\end{equation}
be the empirical distribution at time $t$. Then It{\^o}'s formula shows that
for $f \in C^2(\mathcal{D})$  the following process is a continuous martingale:
\begin{equation}\label{eq:sDK-martingale-problem}
    M^{f}_t = \mu_t (f) - \mu_0(f) - \int_0^t \mu_s (\Delta_G f) \mathd s,
\end{equation}
with quadratic variation
\begin{equation}\label{eq:sDK-martingale-QV}
    \mathd \langle M^{f} \rangle_t = \frac{2}{N^2} \sum_{i = 1}^N | G^{- 1 /
   2} \nabla f (X^i_t) |^2 \mathd t = \frac{2}{N} \mu_t (| G^{- 1 /
   2} \nabla f (X^i_t) |^2) \mathd t.
\end{equation}

While the martingale problem for $\mu$ has a unique solution~\cite{Konarovskyi}, the analysis of numerical approximation schemes naturally leads us to consider the associated nonlinear Dean--Kawasaki equation, which (formally) provides a continuum description of the particle system.
See also [22] for the derivation of the martingale problem if particles are given by independent continuous Markov processes on a general Polish space.


Indeed, let us pretend that $\mu_t$ has a density with respect to Lebesgue measure,
which we denote by $\rho^\lambda_t$. Then
\begin{equation}\label{eq:sDK-Euclidean-f}
    \mu_s(f) = \langle \rho^\lambda_s, f\rangle_{L^2(\mathcal{D})}
\end{equation}
 and
\begin{equation}\label{eq:sDK-Euclidean-Laplacian}
    \mu_s(\Delta_G f) = \langle \rho^\lambda_s, \Delta_G f\rangle_{L^2(\mathcal{D})} = \langle \Delta_G^\ast \rho^\lambda_s, f\rangle_{L^2(\mathcal{D})} ,
\end{equation}
where the adjoint is computed with respect to Lebesgue measure, 
\begin{equation*}
    \Delta_G^{\ast} f = - \nabla \cdummy (b f) + \sqrt{2} D^2 : (G^{- 1} f),
\end{equation*}
with $A:B = \operatorname{Tr}(AB)$. To find a formal representation of the martingale as a stochastic integral, let $W$ be a cylindrical Wiener process with identity covariance on the Hilbert space $U=L^2(\mathcal{D}; \mathbb R^d)$ with inner product $\langle f, g\rangle_U = \int_{\mathcal{D}} f(x) \cdot g(x) \mathd x$, i.e. formally (because $W(t)$ does not actually take values in $U$) it holds $\mathbb E[\langle W(t), f\rangle_U^2] = t \| f\|_U^2$. Then
\begin{eqnarray}\nonumber
  \int_0^t \mu_s (| G^{- 1 / 2} \nabla f (X^i_t) |^2) \mathd s & = & \int_{[0, t]
  \times \mathcal{D}} \rho^{\lambda}_s(x) | G^{- 1 / 2}(x) \nabla f(x) |^2 \mathd s \mathd x,\\ \nonumber
  & = & \int_{0}^t \left\| \sqrt{\rho^{\lambda}_s} G^{- 1 / 2}
  \nabla f \right\|_U^2 \mathd s\\ \nonumber
  & = & \left\langle \int_{0}^\cdot \langle \sqrt{\rho^{\lambda}_s} G^{- 1 / 2}
  \nabla f, \mathd W_s\rangle_U\right\rangle_t\\
  & = & \left\langle \int_{0}^\cdot \left\langle f, \nabla \cdot (\sqrt{\rho^{\lambda}_s} G^{- 1 / 2}\mathd W_s)\right\rangle_{L^2(\mathcal{D})}\right\rangle_t, \label{eq:sDK-Euclidean-Noise}
\end{eqnarray}
where we used that $-\nabla\cdot$ is the adjoint of $\nabla$, as an operator from $U$ to $L^2(\mathcal{D})$.
Combining \eqref{eq:sDK-martingale-problem}, \eqref{eq:sDK-martingale-QV}, \eqref{eq:sDK-Euclidean-f}, \eqref{eq:sDK-Euclidean-Laplacian} and \eqref{eq:sDK-Euclidean-Noise}, we   obtain the Dean--Kawasaki equation for the density $\rho^\lambda$:
\begin{equation}\label{eq:sDK-Euclidean}
    \mathd \rho^\lambda_t = \Delta_G^{\ast} \rho^\lambda_t \mathd t +
   \sqrt{\frac{2}{N}} \nabla \cdummy \left( \sqrt{\rho^\lambda_t} G^{- 1 /
   2} \mathd W_t \right).
\end{equation}

\subsubsection{Intrinsic representation}

We next consider the formulation with respect to the surface area measure 
 $\nu$ given by \eqref{def:nu}. This choice is natural, as it incorporates the geometric structure of the surface. We denote by $\rho^\nu$  the formal density of $\mu$ with respect to $\nu$. Then for
\begin{equation}\label{defHsurface}
    H := L^2(\mathcal{D}, \nu),
\end{equation}
we have
\begin{equation}\label{eq:sDK-Intrinsic-f}
    \mu_s(f) = \langle \rho^\nu_s, f\rangle_{H}
\end{equation}
 and
\begin{equation}\label{eq:sDK-Intrinsic-Laplacian}
    \mu_s(\Delta_G f) = \langle \rho^\nu_s, \Delta_G f\rangle_{H} = \langle \Delta_G \rho^\nu_s, f\rangle_{H},
\end{equation}
where we used that $\Delta_G$ is self-adjoint in $H$. To find a formal representation of the martingale as a stochastic integral, let $W^G$ be a cylindrical Wiener process on the Hilbert space
\begin{equation}\label{def:Usurface}
        U=L^2(\mathcal{D},\nu; \mathbb R^d, (\cdot, \cdot)_G),\qquad \langle f, g\rangle_U = \int_\mathcal{D} (f(x), g(x))_{G(x)} \sqrt{|G(x)|}\mathd x,
\end{equation}
with identity covariance, i.e. formally (because $W^G_t$ does not take values in $U$)
\begin{equation}\label{Wg_stat}
        \mathbb E[\langle W^G_t, f\rangle_U^2] = t \| f\|_U^2 .
\end{equation}
Then, noting that $|G^{-1/2} \nabla f|^2 = \|\nabla_G f\|_G^2$,
\begin{eqnarray}\nonumber
  \int_0^t \mu_s (| G^{- 1 / 2} \nabla f (X^i_t) |^2) \mathd s & = & \int_{0}^t \left\| \sqrt{\rho^{\nu}_s} 
  \nabla_G f \right\|_U^2 \mathd s\\ \nonumber
  & = & \left\langle \int_{0}^\cdot \langle \sqrt{\rho^{\nu}_s} 
  \nabla_G f, \mathd W^G_s\rangle_U\right\rangle_t\\
  & = & \left\langle \int_{0}^\cdot \left\langle f, \operatorname{div}_G (\sqrt{\rho^{\nu}_s} \mathd W^G_s)\right\rangle_{L^2(\mathcal{D})}\right\rangle_t, \label{eq:sDK-Intrinsic-Noise}
\end{eqnarray}
where we used that $-\operatorname{div}_G$ is the adjoint of $\nabla_G$, as an operator from $U$ to $L^2(\mathcal{D})$.
Combining \eqref{eq:sDK-martingale-problem}, \eqref{eq:sDK-martingale-QV}, \eqref{eq:sDK-Intrinsic-f}, \eqref{eq:sDK-Intrinsic-Laplacian} and \eqref{eq:sDK-Intrinsic-Noise}, we   obtain the surface Dean--Kawasaki equation for the density $\rho^\nu$:
\begin{equation}\label{eq:sDK-intrinsic}
    \mathd \rho^\nu_t = \Delta_G \rho^\nu_t \mathd t +
   \sqrt{\frac{2}{N}} \operatorname{div}_G \left( \sqrt{\rho^\nu_t} \mathd W^G_t \right).
\end{equation}

We note that this is essentially the same equation as (1.10) in \cite{DelloSchiavo2024massive}, which formulates a related result for independent particles on Dirichlet spaces via Hilbert tangent bundles.

\begin{remark}[On the interpretation of $W^G$]\label{rem:stwn}
Note that statistically $W^G= |G|^{-1/4}G^{-1/2}W$, where $W$ is a standard $\mathbb{R}^d$-valued space-time noise. Indeed, for $f\in U$ it holds
\[
    \mathbb{E}\left[\left<|G|^{-1/4}G^{-1/2}W_t, f\right>^2_U\right] = 
    \mathbb{E}\left[\left<W_t, |G|^{1/4}G^{1/2}f\right>^2_{L^2(\mathcal{D}, \mathbb{R}^d)}\right]
    =t \||G|^{1/4}G^{1/2}f\|^2_{L^2(\mathcal{D}, \mathbb{R}^d)}=t\|f\|^2_U,
\]
which coincides with \eqref{Wg_stat}.
\end{remark}

Starting from the intrinsic formulation of the surface Dean–Kawasaki Equation~\eqref{eq:sDK-intrinsic}, together with the representation of the noise and the definitions of the metric differential operators \eqref{nablaG}, \eqref{def:divG} and \eqref{def:LB}, we obtain a coordinate representation of the equation in terms of local coordinates and standard Euclidean differential operators: 
\begin{eqnarray}\label{eq:numDK}
    \mathd\rho^\nu_t = \frac{1}{\sqrt{|G|}} \nabla \cdot (\sqrt{|G|} G^{-1}  \nabla \rho^\nu_t) \mathd t + \sqrt{\frac{2}{N}} \frac{1}{\sqrt{|G|}}\nabla \cdot (\sqrt{\rho^\nu_t} | G |^{1 / 4} G^{- 1 / 2} \mathd W_t ),
\end{eqnarray}
where $W$ is the usual $\mathbb{R}^d$-valued space-time white noise. This formulation will serve as the basis for the numerical discretization.

\begin{remark}
    If we start from the Brownian particles that are given via Stratonovich formulation on $\Gamma$, as described in Remark \ref{rem:StrBM}, the martingale formulation of the Dean--Kawasaki equation for the corresponding empirical distribution $\mu^{\Gamma}$ is
    \[
    \mathd\mu_t^{\Gamma}(f)=\mu_t^{\Gamma}(\Delta_\Gamma f) \mathd t + \mathd M_t^f,
    \]
    where $\Delta_\Gamma = \operatorname{div}_\Gamma \nabla_\Gamma$, with $\nabla_\Gamma$ as in~\eqref{eq:tangential-gradient}, with $-\operatorname{div}_\Gamma$ the adjoint of $\nabla_\Gamma$ as an operator from $L^2(\Gamma, S)$ to $L^2(\Gamma, S; T\Gamma)$ (recall that the surface measure $S$ is defined in \eqref{partial_int}), and
    where $M^f$ is a continuous martingale with quadratic variation
    \[
     \mathd \langle M^{f} \rangle_t = \frac{2}{N}\mu^{\Gamma}_t(|\nabla_\Gamma f|^2)\mathd t .
    \]
    Next, if we pretend that $\mu^\Gamma$ has a density $\rho^S$ with respect to the surface measure $S$ defined in \eqref{partial_int}, then for
    \[
        H:=L^2(\Gamma,S) \quad  \text{and}\quad U := L^2(\Gamma,S;T\Gamma),
    \]
    and for $W^S$ a cylindrical Wiener process with identity covariance on $U$,
    \[
        \mathbb{E}[\langle W_t^S, f\rangle_U^2]=t\|f\|_U^2,
    \]
    we have
    \[
        \mu^{\Gamma}_t(|\nabla_\Gamma f|^2) \mathd t = \|\sqrt{\rho_t^S}\nabla_\Gamma f\|^2_U \mathd t =\mathd \langle\int_0^\cdot \langle\sqrt{\rho_s^S}\nabla_\Gamma f, \mathd W_s^S\rangle_U
\rangle_t=\langle\int_0^\cdot \langle f, \operatorname{div}_\Gamma \left(\sqrt{\rho_s^S}\mathd W_s^S \right)\rangle_H\rangle_t.
    \]
Hence, using also that $\Delta_\Gamma$ is self-adjoint in $L^2(\Gamma, S)$, the surface Dean--Kawasaki equation for $\rho^S$ is
\[
\mathd \rho^S_t=\Delta_\Gamma \rho_t^S \mathd t +\sqrt{\frac2N}\operatorname{div}_\Gamma\left( \sqrt{\rho^S_t}\mathd W_t^S\right).
\]
  
\end{remark}

\subsubsection{
Fluctuation dissipation  relation}

The  linearization of the intrinsic surface Dean--Kawasaki Equation~\eqref{eq:sDK-intrinsic} around its time-stationary mean field limit
\[
    \overline{\rho}^\nu \equiv \frac{1}{\nu(\mathcal{D})}
\]
is given by the Ornstein-Uhlenbeck process
\begin{equation}\label{eq:lin}
       \mathd Z_t = \Delta_G Z_t \mathd t +
   \sqrt{\frac{2}{N}} \operatorname{div}_G \left( \sqrt{\overline{\rho}^\nu} \mathd W^G_t \right).
\end{equation}
Then for the spaces $H$ and $U$ defined by \eqref{defHsurface} and \eqref{def:Usurface} respectively, the operator $A:=\Delta_G$ is a self-adjoint operator on $H$, the cylindrical Wiener process $W^G$ on $U$ has covariance $Q:=\operatorname{Id}$, and $B:=\sqrt{\frac{2}{N}\overline{\rho}^\nu}\operatorname{div}_G$ is a linear operator from $U$ to $H$ with adjoint $B^*=-\sqrt{\frac{2}{N}\overline{\rho}^\nu}\nabla_G$. 
By \cite{DaPrato2014}, Section~11.3, the covariance $C$ of the centered Gaussian equilibrium distribution satisfies 
\begin{equation*}
    AC+CA=-BQB^*,
\end{equation*}
which implies that
\[
C = \frac{\overline{\rho}^\nu}{N}\operatorname{Id} ,
\]
which matches the statistics of the coarse-grained particle system and confirms the choice of the fluctuating noise in \eqref{eq:sDK-intrinsic}.
Hence, the fluctuation dissipation relation is fulfilled. 

Furthermore, the formal covariance function of $Z$ at stationarity is given by
\begin{equation}\label{cFDR}
    \mathbb{E}[Z(x)Z(y)]=\frac{\overline{\rho}^\nu}{N\sqrt{|G|(x)}}\delta(x-y),
\end{equation}
which we want to preserve under numerical discretization.

\subsection{Surface Dean--Kawasaki equation for interacting particle system on fixed hypersurface}

We introduce potentials that lead to invariant Gibbs measures: Let $V \in
C^2_b (\mathcal{D} ; \mathbb{R})$ and let $U \in C^2_b (\mathcal{D} \times \mathcal{D} ; \mathbb{R})$ be
symmetric and consider the following (non-normalized) measure on $D^N$
\[ \nu_N (\mathd x) := \exp \left( - \sum_i V (x_i) - \frac{1}{2 N} \sum_{i,
   j} U (x_i, x_j) \right) \prod_{i = 1}^N \sqrt{| G | (x_i)} \mathd x_i . \]
Our goal is to find a reversible diffusion process $(X^1, \ldots, X^N)$ with invariant measure $\nu_N$, in analogy with the usual Langevin diffusion. This
will be based on the following computation.

\begin{remark}\label{rem:LR}
  The operator $\operatorname{div}_G$ satisfies the following Leibniz type rule:
  \[ \operatorname{div}_G (f g) = \frac{1}{\sqrt{| G |}} \nabla \cdummy \left(
     \sqrt{| G |} f g \right) = \frac{1}{\sqrt{| G |}} \nabla \cdummy \left(
     \sqrt{| G |} f \right) g + \frac{1}{\sqrt{| G |}} \sqrt{| G |} f \nabla
     \cdummy g = (\operatorname{div}_G f) g + f \cdummy \nabla g. \]
\end{remark}
  For arbitrary  $F \in C^2 (\mathcal{D}^N)$, the operator 
  \[
  \mathcal{L}_{N, F} = e^F \sum_{i = 1}^N \operatorname{div}_{G, x_i} (e^{- F} \nabla_{G, x_i}) , 
  \]
 is  symmetric with respect to the measure $e^{- F (x)} \prod_{i =
  1}^N \sqrt{| G | (x_i)} \mathd x_i$. 
  Using the Leibniz rule from Remark \ref{rem:LR}, the operator $\operatorname{div}_{G}$ can be rewritten as follows:
  \[ \mathcal{L}_{N, F}  = \sum_{i = 1}^N (\Delta_{G, x_i} - \nabla_{x_i} F
     \cdummy \nabla_{G, x_i}) = \sum_{i = 1}^N (\Delta_{G, x_i} - \nabla_{G,
     x_i} F \cdummy \nabla_{x_i}) . \]
  Choosing $F (x) = \sum_k V (x_k) + \frac{1}{2 N} \sum_{k, \ell} U (x_k,
  x_{\ell})$, we obtain
  \[ \nabla_{x_i} F (x) = \nabla V (x_i) + \frac{1}{2 N} \sum_{\ell}
     \nabla_{x_i} U (x_i, x_{\ell}) + \frac{1}{2 N} \sum_k \nabla_{x_i} U
     (x_k, x_i) = \nabla V (x_i) + \frac{1}{N} \sum_j \nabla_{x_i} U (x_i,
     x_j), \]
  where we used that $U$ is symmetric.

This leads to the definition of a Langevin diffusion for $\nu_N$:

\begin{definition}
  We call a stochastic process $(X^1, \ldots, X^N)$ with values in $\mathcal{D}^N$ the
  {\tmem{Langevin diffusion with invariant measure $\nu_N$}}, if it is a
  Markov process with generator
  \begin{eqnarray*}
    \mathcal{L}_N & = & \sum_{i = 1}^N \left( - \left( \nabla_G V (x_i) +
    \frac{1}{N} \sum_{j = 1}^N \nabla_{G, x_i} U (x_i, x_j) \right) \cdummy
    \nabla_{x_i} + \Delta_{G, x_i} \right)\\
    & = & \sum_{i = 1}^N \left( - \left( \nabla_G V (x_i) + \frac{1}{N}
    \sum_{j = 1}^N \nabla_{G, x_i} U (x_i, x_j), \nabla_{G, x_i} \right)_G +
    \Delta_{G, x_i} \right).
  \end{eqnarray*}
\end{definition}

The SDE for this generator is
\begin{eqnarray}\label{eq:Lang}
  \mathd X^i_t & = & b (X^i_t) \mathd t - \nabla_G V (X^i_t) \mathd t -
  \frac{1}{N} \sum_{j = 1}^N \nabla_{G, x_i} U (X^i_t, X^j_t) \mathd t +
  \sqrt{2} G^{- 1 / 2} (X_t^i) \mathd B^i_t \\
  & = & b (X^i_t) \mathd t - \nabla_G V (X^i_t) \mathd t - \int_D \nabla_{G,
  x_i} U (X^i_t, x) \mu_t (\mathd x) \mathd t + \sqrt{2} G^{- 1 / 2} (X_t^i)
  \mathd B^i_t,\nonumber 
\end{eqnarray}
where $b$ is defined by \eqref{eq:BMonSurface}. Let $\mu$ be the empirical distribution. In a similar way as for \eqref{eq:sDK-Euclidean}, we derive the  Dean--Kawasaki equation for the formal density $\rho^\lambda$ of $\mu$ with respect to Lebesgue measure:
\begin{equation*}
    \mathd \rho_t^\lambda = \Delta_G^{\ast} \rho_t^\lambda \mathd t +
   \nabla \cdummy \left( \left( \nabla_G V + \int_D \nabla_G U (\cdummy, x)
   \rho_t^\lambda(x) \mathd x \right) \rho_t^\lambda \right) \mathd t +
   \sqrt{\frac{2}{N}} \nabla \cdummy \left( \sqrt{\rho_t^\lambda} G^{- 1 /
   2} \mathd W \right) . 
\end{equation*} 

The surface Dean--Kawasaki equation in the intrinsic form for the (formal) density $\rho^\nu$ with respect to the surface area element $\nu$ is
\begin{equation}\label{eq:sDK-nu-potential}
    \mathd\rho^\nu_t = \Delta_G \rho^\nu_t \mathd t+  \text{div}_G \left(\left(\nabla_G V +\int_D \nabla_G U (\cdummy, x)
   \rho_t^\nu(x) \mathd x \right)\rho^\nu_t\right) \mathd t+ \sqrt{\frac{2}{N}}\text{div}_G(\sqrt{\rho^\nu}\mathd W^G).
\end{equation}

\subsection{The surface Dean--Kawasaki equation on a moving, interacting surface.}

We now consider a more general setting in which the hypersurface is moving and the induced metric $G$ is governed by a stochastic process
$\eta$ and is therefore time-dependent.
 Furthermore, we allow the particles and the hypersurface to interact. A typical example of a moving surface is given by the ultraviolet-truncated Langevin dynamics of a Helfrich elastic membrane, which can be represented as a multidimensional Ornstein--Uhlenbeck process; see \cite[Section 4]{DuncanElliottPavliotisStuart}
for further details.

 More precisely,  we assume that $\eta$ solves another stochastic differential equation,
coupled with the empirical measure of the particle system, and that $G_t = G (\cdummy, \eta_t)$:
\begin{equation}\label{eq:eta}
    \mathd \eta_t = a (\eta_t, \mu_t) \mathd t + \sigma (\eta_t, \mu_t)
   \mathd B_t,
\end{equation}
 Moreover, we assume
that
\[
    \mathd X^i_t = b_t (X^i_t) \mathd t - \nabla_{G_{t}} V (X^i_t)
   \mathd t - \int_\mathcal{D} \nabla_{G_{t}} U (X^i_t, x) \mu_t (\mathd x) \mathd
   t + \sqrt{2} G^{- 1 / 2}_{t} (X_t^i) \mathd B^i_t,
\]
where $B$ is independent from $(B^1, \ldots, B^N)$, and
\[
  b_t (x) = \frac{1}{\sqrt{| G_{t} | (x)}} \nabla_x \cdummy
  \left( \sqrt{| G_{t} | (x)} G^{- 1}_{t} (x) \right).
\] 

Then, the Dean--Kawasaki equation for the formal density $\rho^\lambda$ of the empirical distribution with respect to Lebesgue measure is given by 
\begin{eqnarray*}
  \mathd \rho^\lambda_t & = & \Delta_{G_{t}}^{\ast} \rho^\lambda_t \mathd t + \nabla \cdummy \left( \left( \nabla_{G_{t}} V +
  \int \nabla_{G_{t}} U (\cdummy, x) \rho^\lambda_t (x) \mathd x
  \right) \rho^\lambda \right) \mathd t\\
  &  & + \sqrt{\frac{2}{N}} \nabla \cdummy \left( \sqrt{\rho^\lambda}
  G_{t}^{- 1 / 2} \mathd W_t \right),\\
  \mathd \eta_t & = & a (\eta_t,\rho_t^\lambda) \mathd t + \sigma
  (\eta_t, \rho^\lambda_t) \mathd B_t,
\end{eqnarray*}
where the space-time white noise $W$ is independent of $B$, the independence being inherited from that of $(B^1,\dots,B^N)$ with $B$.

Next, we derive the equation for the formal density $\rho^{\nu_t}$ of $\mu_t$ with respect to
\begin{equation*}
    \nu_t (\mathd x)  := g_t(x)\mathd x:=\sqrt{| G_t (x) |} \mathd x. 
\end{equation*}
   
Since formally  $\rho^{\nu_t} \nu_t = \mu_t$, from It\^o's formula we have
\begin{equation}\label{eq:all-moving-prelim}
    \mathd \rho^{\nu_t} \nu_t (f) = \rho^{\nu_t} \nu_t \left( \Delta_{G_t} f -
   \left( \nabla_{G_t} V + \int_\mathcal{D} \nabla_{G_{t}} U (\cdummy, x) \rho^{\nu_t}_t
   \nu_t (\mathd x), \nabla_{G_{t}} f \right)_{G_t} \right) \mathd t + \mathd
   M^f_t,
\end{equation}
with a continuous martingale $M^f$ that satisfies
\[ \mathd \langle M^f_t \rangle = \frac2N \rho^{\nu_t}_t \nu_t (| \nabla_{G_t} f
   |_{G_t}^2)\mathd t, \qquad \mathd \langle M^f, B^\eta \rangle_t = 0. \]
Thus, we can represent the martingale as follows: Let $W$ be a space-time
white noise in $L^2 (D ; \mathbb{R}^d)$, independent of $B^\eta$, and let
\begin{equation}
    W^G_t := \int_0^t g_s^{- 1 / 2} G_s^{- 1 / 2} \mathd W_s . 
\end{equation}
Let also
\[ U_t := L^2 (\mathcal{D}, \nu_t ; \mathbb{R}^d, (\cdummy, \cdummy)_{G_t}), \]
so that formally for any adapted process $u$ with $\int_0^T \| u_t \|_{U_t}^2 \mathd t < \infty$ for all $T>0$:
\[ \mathd \left\langle \int_0^{\cdummy} \langle u_s, \mathd W^G_s \rangle_{U_s}
   \right\rangle_t = \mathd \left\langle \int_0^{\cdummy} \langle g_t^{1 / 2}
   G_t^{1 / 2} u_s, \mathd W_s \rangle_{L^2 (D ; \mathbb{R}^d)} \right\rangle_t
   = \| g_t^{1 / 2} G_t^{1 / 2} u_t \|_{L^2 (D ; \mathbb{R}^d)}^2 \mathd t = \|
   u_t \|_{U_t}^2 \mathd t. \]
Then we have with  $H_t := L^2 (\mathcal{D}, \nu_t)$
\begin{align*}
  \frac{N}{2}\mathd \langle M^f_t \rangle & = \rho^{\nu_t}_t \nu_t (| \nabla_{G_t} f
  |_{G_t}^2) \mathd t = \left\| \sqrt{\rho^{\nu_t}_t} \nabla_{G_t} f \right\|_{U_t}^2 \mathd
  t = \mathd \left\langle \int_0^{\cdummy} \left\langle
  \sqrt{\rho^{\nu_t}_t} \nabla_{G_s} f, \mathd W^G_s \right\rangle_{U_s}
  \right\rangle_t\\
  & = \mathd \left\langle \int_0^{\cdummy} \left\langle f, \tmop{div}_{G_s}
  \left( \sqrt{\rho^{\nu_s}_s} \mathd W^G_s \right) \right\rangle_{H_s}
  \right\rangle_t.
\end{align*}
Now we  integrate $\Delta_{G_t}$ and $\nabla_{G_t}$ in \eqref{eq:all-moving-prelim} by parts to obtain the equation
\begin{align*}
  \mathd \langle \rho^{\nu_t}_t, f \rangle_{H_t} & = \left\langle
  \Delta_{G_t} \rho^{\nu_t}_t + \tmop{div}_{G_t} \left( \left(
  \nabla_{G_t} V + \int \nabla_{G_t} U (\cdummy, x) \rho^{\nu_t}_t (x) \nu_t
  (\mathd x) \right) \rho^{\nu_t}_t \right), f \right\rangle_{H_t} \mathd t\\
  &\quad +\sqrt{\frac2N} \left\langle f, \tmop{div}_{G_t} \left( \sqrt{\rho^{\nu_t}} \mathd
  W^G_t \right) \right\rangle_{H_t},
\end{align*}
which we can rewrite as
\begin{align*}
    \mathd (\rho^{\nu_t}_t g_t) & = \left( \Delta_{G_t} \rho^{\nu_t}_t +
   \tmop{div}_{G_t} \left( \left( \nabla_{G_t} V + \int \nabla_{G_t} U
   (\cdummy, x) \rho^{\nu_t}_t (x) \nu_t (\mathd x) \right) \rho^{\nu_t}_t
   \right) \right) g_t \mathd t \\
   &\quad + \sqrt{\frac2N} g_t \tmop{div}_{G_t} \left(
   \sqrt{\rho^{\nu_t}} \mathd W^G_t \right) .
\end{align*}
Since this is not a closed equation,  we  divide both sides by $g_t$. Using that
\[ \mathd g_t^{- 1} = - g_t^{- 2} \mathd g_t + g_t^{- 3} \mathd \langle g
   \rangle_t, \qquad \mathd \left\langle g^{- 1}, \int_0^{\cdummy} f_s W^G_s
   \right\rangle_t = 0, \]
we obtain the interacting surface Dean--Kawasaki equation on a time-dependent surface with particle-surface coupling:
\begin{align} \nonumber
  \mathd \rho^\nu_t & = \mathd (\rho^{\nu_t}_t g_t) g_t^{- 1}\\ \nonumber
  & = g_t^{- 1} \mathd (\rho^{\nu_t}_t
  g_t) + (\rho^{\nu_t}_t g_t) \mathd g_t^{- 1} + \mathd \langle g^{- 1},
  \rho^{\nu} g \rangle_t\\ \nonumber
  & = \left( \Delta_{G_t} \rho^{\nu_t}_t + \operatorname{div}_{G_t} \left(
  \left( \nabla_{G_t} V + \int \nabla_{G_t} U (\cdummy, x) \rho^{\nu_t}_t (x)
  \nu_t (\mathd x) \right) \rho^{\nu_t}_t \right) \right) \mathd t \\
  & \quad + \sqrt{\frac2N}\operatorname{div}_{G_t} \left( \sqrt{\rho^{\nu_t}} \mathd W^G_t \right)
   - \rho^{\nu_t}_t (g_t^{- 1} \mathd g_t - g_t^{- 2} \mathd \langle g \label{eq:sDK-full}
  \rangle_t) .
\end{align}

The rigorous formulation of the surface Dean--Kawasaki equation is given in terms of a martingale problem, similarly to \cite{konarovskyi2020dean,Konarovskyi}. In the following definition, $\mathcal{P}(\mathcal{D})$ denotes the probability measures on $\mathcal{D}$.

\begin{definition}[Martingale problem]\label{def:sDK-martingale-problem}
  A stochastic process $(\mu, \eta)$ in $C (\mathbb{R}_+, \mathcal{P} (\mathcal{D})
  \times \mathbb{R}^e)$ is called a martingale solution to the
  sDK equation \eqref{eq:sDK-full}, if $\eta$ is a (probabilistically) weak  solution to \eqref{eq:eta}
  and for all $f \in C^2 (\mathcal{D})$ the process
  \[ M^f_t = \mu_t (f) - \mu_0 (f) - \int_0^t \mu_s \left( \Delta_{G_s} f -
     \nabla_{G_s} V \cdummy \nabla f - \int \nabla_{G_s} U (\cdummy, x) \mu_s
     (\mathd x) \cdummy \nabla f \right) \mathd s, \]
  $t\ge 0$, is a continuous martingale with
  \[ \langle M^f \rangle_t = \frac{2}{N} \int_0^t \mu_s (| \nabla_{G_s} f
     |_{G_s}^2) \mathd s, \qquad \langle M^f, B \rangle_t = 0, \qquad t
     \geq 0. \]
\end{definition}

As usual, in the regime where the martingale problem is formally associated to the empirical distribution of interacting particles coupled to a fluctuating surface,  existence of solutions to the martingale problem follows from It\^o's formula provided the particle-surface system admits weak solutions. In different parameter regimes (prefactor $\frac2\alpha$ instead of $\frac2N$, with $\alpha \not\in \mathbb N$, or $\mu_0$ not of the form $\frac1N \sum_{i=1}^N \delta_{x_i}$), there might not exist any solutions to the martingale problem, in the spirit of the non-existence result of \cite{Konarovskyi}.

If $U = 0$ and if $a$ and $\sigma$ are independent of $\mu$, we
obtain the weak uniqueness of solutions to the martingale problem. This result is analogous to \cite{Konarovskyi}, although our proof is based on the well-posedness of the moment problem for $\mu$, rather than exact duality, and therefore it seems slightly more robust. As in \cite{konarovskyi2020dean}, this could be extended to $U\neq 0$ and to $a$ which depends on $\mu$, as long as $U$ and $a$ can be removed by a Girsanov transform.

\begin{theorem}[Weak uniqueness]
Let $U = 0$ and let $a$ and $\sigma$ be independent of $\mu$. Let
$\mathcal{G}_t =\mathcal{F}_t \vee \mathcal{F}^B_{\infty}$, $t\ge 0$, where
$\mathcal{F}^B$ is the canonical filtration generated by $B$. We assume that
for all $\mathcal{G}_0$-measurable $f \in C^2 (\mathcal{D})$ and $t \geq 0$ there
exists a $\mathcal{G}_0$-measurable solution $P_{\cdummy, t} f \in C^{1, 2}
(\mathbb{R}_+ \times \mathcal{D})$ to the random Kolmogorov backward equation
\[ \partial_s P_{s, t} f = - (\Delta_{G_s} P_{s, t} f - \nabla_{G_s} V \cdummy
   \nabla P_{s, t} f), \qquad P_{t, t} f = f, \]
and that strong existence and weak uniqueness hold for~\eqref{eq:eta} (which is in particular the case if $a$ and $\sigma$ are Lipschitz continuous). Then the law of the martingale solution $(\mu, \eta)$ to
the sDK Equation~\eqref{eq:sDK-full} is uniquely determined by its initial distribution.
\end{theorem}

\begin{proof}
    We start by showing that $\mathbb{E} [F (\mu_t) \mid \mathcal{G}_0]$ is uniquely
determined by $\tmop{law} (\mu_0)$ and the path $(\eta_s)_{s \geq 0}$,
for all $t \geq 0$ and $F \in C (\mathcal{P} (D))$. By a density
argument and using the linearity of $f \mapsto \mu_t (f)$, it suffices to
consider $F (\mu) = \varphi (\mu_t (f))$ for $f \in C^2$, and it is
convenient to allow $\mathcal{G}_0$-measurable random $f \in C^2$. Since
$\mu_t$ is a probability measure for each $t \geq 0$, the random variable
$\mu_t (f)$ is bounded and thus we can restrict to $\varphi (x) = x^m$ by the
Stone-Weierstra{\ss} theorem. The lemma following below this proof shows that the martingale $M^f$ in Definition~\ref{def:sDK-martingale-problem} remains a martingale in the filtration $\mathcal G$. We combine this with the telescope sum $\mu_t (P_{t, t} f) - \mu_0 (P_{0, t} f) = \lim_n \sum_{k=0}^{n-1}(\mu_{(k+1)t/n} (P_{(k+1)t/n, t} f) - \mu_{kt/n} (P_{kt/n, t} f))$ to show that for all
$\mathcal{G}_0$-measurable random $f$
\[ \mu_t (P_{t, t} f) = \mu_0 (P_{0, t} f) + \int_0^t \mu_s \left(  \left(
   \partial_s + \Delta_{G_{\eta_s}} - \nabla_{G_{\eta_s}} V \cdummy \nabla
   \right) P_{s, t} f \right) \mathd s + M^{P f}_t = \mu_0 (P_{0, t} f) + M^{P
   f}_t, \]
where $M^{P f}$ is a continuous martingale in the filtration
$(\mathcal{G}_t)_{t \geq 0}$, with quadratic variation
\[ \langle M^{P f} \rangle_s = \frac{2}{N} \int_0^s \mu_r (| \nabla_{G_r}
   P_{r, t} f |^2_{G_r}) \mathd r. \]
In particular, we obtain from It{\^o}'s formula
\begin{align} \nonumber
  & \mathbb{E} [(\mu_t (f)^m - \mu_0 (P_{0, t} f)^m) \mid \mathcal{G}_0]\\  \nonumber
  & =\mathbb{E} \left[ \left. \int_0^t m \mu_s (P_{s, t} f)^{m - 1} \mathd
  M^{P f}_s + \frac{m (m - 1)}{N} \int_0^t \mu_s (P_{s, t} f)^{m - 2} \mu_s (|
  \nabla_{G_s} P_{s, t} f |^2_{G_s}) \mathd s \,\right|\, \mathcal{G}_0 \right]\\ \label{eq:weak-uniqueness-pr1}
  & = \frac{m (m - 1)}{N} \int_0^t \mathbb{E} [\mu_s (P_{s, t} f)^{m - 2}
  \mu_s (| \nabla_{G_s} P_{s, t} f |^2_{G_s}) \mid \mathcal{G}_0] \mathd s.
\end{align}
Observe that the right hand side is $(m - 1)$-linear in $\mu_s$, which allows
an inductive proof of uniqueness: For $m = 1$, Equation~\eqref{eq:weak-uniqueness-pr1} yields
\[ \mathbb{E} [\mu_t (f) \mid \mathcal{G}_0] =\mathbb{E} [\mu_0 (P_{0, t} f)
   \mid \mathcal{G}_0], \qquad t \geq 0. \]
Assume now by induction that for all $t \geq 0$ and all
$\mathcal{G}_0$-measurable random $f_1, \ldots, f_{m - 1} \in C^2$ the
conditional expectation $\mathbb{E} [\mu_t (f_1) \cdots \mu_t (f_{m - 1})
\mid\mathcal{G}_0]$ is determined by the initial distribution of $\mu$. Then \eqref{eq:weak-uniqueness-pr1} shows that also $\mathbb{E} [\mu_t (f)^m
\mid\mathcal{G}_0]$ is uniquely determined. By
polarization, also $\mathbb{E} [\mu_t (f_1) \cdots \mu_t (f_m)
\mid\mathcal{G}_0]$ is uniquely determined for all $t \geq 0$, $m \in
\mathbb{N}$, and $\mathcal{G}_0$-measurable random $f_1, \ldots, f_m \in C^2$, which concludes the induction step.

By the tower property together with a monotone class argument, we then get
uniqueness of $\mathbb{E} [\Phi (\mu_t, \eta_t)]$ for all bounded and
measurable $\Phi : \mathcal{P} (E) \times \mathbb{R}^e \rightarrow
\mathbb{R}$. Now we perform another induction over the set of times $t_1,
\ldots, t_n$ to show that
\[ \mathbb{E} [\Phi_1 (\mu_{t_1}, \eta_{t_1}) \cdots \Phi_n (\mu_{t_n},
   \eta_{t_n})] \]
is uniquely determined by the initial distribution. We have just shown the
case $n = 1$. For general $n$, we condition on $\mathcal{F}_{t_{n - 1}}$ and
apply the previous argument to show that $\mathbb{E} [\Phi (\mu_{t_n},
\eta_{t_n}) \mid\mathcal{F}_{t_{n - 1}}]$ is uniquely determined by $ (\mu_{t_{n
- 1}}, \eta_{t_{n - 1}})$, i.e.
\[ \mathbb{E} [\Phi_1 (\mu_{t_1}, \eta_{t_1}) \cdots \Phi_{n - 1} (\mu_{t_{n -
   1}}, \eta_{t_{n - 1}}) \mathbb{E} [\Phi_n (\mu_{t_n}, \eta_{t_n})
   \mid\mathcal{F}_{t_{n - 1}}]] =\mathbb{E} [\Phi_1 (\mu_{t_1}, \eta_{t_1})
   \cdots \tilde{\Phi}_{n - 1} (\mu_{t_{n - 1}}, \eta_{t_{n - 1}})], \]
so the claim for $n$ follows from that for $n-1$. We conclude by another monotone class argument
to get the uniqueness of finite-dimensional distributions.
\end{proof}

In the proof we used the following presumably classical lemma, whose proof is in~\ref{app:aux-lemma}.

\begin{lemma}\label{lem:martingale-filtration}
  Let $(\mathcal{F}_t)_{t \geq 0}$ be a filtration and let $M, B$ both be
  martingales in $\mathcal{F}$, such that $B$ is a $d$-dimensional Brownian
  motion and $\langle M, B^i \rangle \equiv 0$ for $i \in \{ 1, \ldots, d \}$.
  Then $M$ is a martingale in the larger filtration $\mathcal{G}_t
  =\mathcal{F}_t \vee \mathcal{F}^B_{\infty}$, $t \ge 0$, where $\mathcal{F}^B_{\infty} =
  \sigma (B_s : s \geq 0)$.
\end{lemma}

\section{Discretization}\label{sec:discretization}

In this section we present a finite volume discretization of the surface Dean--Kawasaki equation given by \eqref{eq:sDK-nu-potential} in two dimensions with possible external potential, for the formal density of the empirical measure with respect to the surface area element. Note that, in order to simplify the notation, we will no longer use superscripts to indicate the reference measure, as in the previous sections, and simply write $\rho$. 

The discretization is constructed on the domain $\mathcal{D}$ using the
representation of the surface Dean--Kawasaki equation in local coordinates and Euclidean differential operators:
\begin{align}
    \mathd\rho =\frac{1}{\sqrt{|G|}}\nabla \cdot (\sqrt{|G|}G^{-1}\nabla \rho)\mathd t \ + \ & \frac{1}{\sqrt{|G|}}\nabla \cdot ( \sqrt{|G|} \rho G^{-1}\nabla V)\mathd t \nonumber \\ +\ & \sqrt{\frac{2}{N}} \frac{1}{\sqrt{|G|}} \nabla \cdot (\sqrt{\rho} |G|^{1/4}G^{-1/2}\mathd W) .
    \label{eq:instrinic_form}
\end{align}

We assume that $\mathcal{D}$ is rectangular of dimension $L_x \times L_y$ with periodic boundary conditions.
We introduce a uniform finite volume mesh of size $I\times J$  with mesh spacing $\Delta x = \frac{L_x}{I}$ and $\Delta y = \frac{L_y}{J}$.  We then let
\begin{align*}
x_{i} =&( i+\frac{1}{2}) \Delta x, \;\;\;\;i= 0,...,I-1, \\
y_{j} =&( j+\frac{1}{2}) \Delta y, \;\;\;\;j= 0,...,J-1,
\end{align*}
and we define cell  $C_{i,j}$ to be
\[
C_{i,j} = [x_{i-\half},x_{i+\half}] \times  [y_{j-\half},y_{j+\half}],
\]
where $x_{i-\half} = x_i - \frac{\Delta x}{2}$, etc.

We approximate $\rho$ by its average value over each cell using
\[
\rho_{i,j} (t)
 \approx  \frac {\int_{C_{i,j}} \rho(x,t) \jacdat \mathd x }{\int_{C_{i,j}} \jacdat \mathd x},
\]
and we let $\rho_{i,j}^n \approx \rho_{i,j}(n \Delta t)$, where $\Delta t$ is the time step.

The presence of off-diagonal terms in $G^{-1}$ introduces tradeoffs in the construction of a discretization.  In particular, a standard finite volume discretization of the Laplace-Beltrami operator results in a semi-discrete system that fails to satisfy the fluctuation dissipation balance.  Here we construct a discretization that satisfies a discrete fluctuation dissipation balance.
We first define a discrete gradient by
\[
\nabla_h^- \rho_{i,j} := \Biggl( 
\begin{matrix}
    \frac{\rho_{i,j}-\rho_{i-1,j}}{\Delta x} \\[.2cm]
    \frac{\rho_{i,j}-\rho_{i,j-1}}{\Delta y}
\end{matrix}\Biggr).
\]
We next define the discrete divergence operator, $\nabla_h^+ \cdot$, acting on a discrete vector field, $(u_{i,j},v_{i,j})^T$, by taking the discrete adjoint of $\nabla_h^-$ to obtain
\[
\nabla_h^+ \cdot
\begin{pmatrix}
u_{i,j} \\
v_{i,j} 
\end{pmatrix}
 := 
\frac{u_{i+1,j}-u_{i,j}}{\Delta x} + \frac{v_{i,j+1}-v_{i,j}}{\Delta y} .
\]
Using these definitions, we obtain the semi-discrete system of stochastic ODEs
\begin{eqnarray}
    d\rho_{i,j} =& \!\!\!\frac{1}{\sqrt{|G|}_{i,j}}\nabla_h^+ \cdot (\sqrt{|G|}_{i,j}G_{i,j}^{-1}\nabla_h^- \rho_{i,j})\mathd t + \frac{1}{\sqrt{|G|}_{i,j}}\nabla_h^+ \cdot ( \sqrt{|G|}_{i,j} \rho_{i,j} G_{i,j}^{-1}\nabla_h^- V_{i,j}) \mathd t \nonumber \\ &+ \sqrt{\frac{2}{N}} \frac{1}{\sqrt{|G|}_{i,j}} \nabla_h^+ \cdot \left(\sqrt{\rho}_{i,j} |G|_{i,j}^{1/4}G_{i,j}^{-1/2} \frac{1}{\sqrt{ \Delta x \Delta y}} \begin{pmatrix}
        \mathd\mathrm{W}_{i,j}^x \\[.1cm]
        \mathd\mathrm{W}_{i,j}^y
    \end{pmatrix}\right),
    \label{eq:fv_semidisc}
\end{eqnarray}
where the $(\mathrm{W}_{i,j}^{x,y})_{i,j,x,y}$ are a collection of independent standard two-dimensional Brownian motions.  In this discretization, all terms involving the metric tensor, $G$, and the potential, $V$, are evaluated at the cell center using analytic formulae.  We note that $G_{i,j}^{-1/2}$ is not uniquely defined.  Any choice of $G_{i,j}^{-1/2}$ that satisfies $G_{i,j}^{-1/2} \, (G_{i,j}^{-1/2})^T = G_{i,j}^{-1}$ will yield the same behavior.

The semidiscrete system Eq. (\ref{eq:fv_semidisc}) with $V\equiv0$ can be written in matrix form as
\begin{eqnarray}\label{num_rho}
    \mathd\boldsymbol{\rho}&=& \mathbf{J}^{-1} \boldsymbol{\nabla}_h^+ (\mathbf{J} 
    \mathbf{G}^{-1}  \boldsymbol{\nabla}_h^-\boldsymbol{\rho})\, \mathd t \nonumber \\
    && + \sqrt{\frac{2}{N \Delta x \Delta y}} \, \mathbf{J}^{-1}
    \boldsymbol{\nabla}_h^+  \left( \mathrm{Diag}(\boldsymbol{\sqrt{\rho}}) \,\mathbf{J}^{1/2} \mathbf{G}^{-1/2}
    \begin{pmatrix}
        \mathd\mathbf{W}^x \\[.1cm]
        \mathd\mathbf{W}^y
    \end{pmatrix}\right),
    \label{eq:fv_matdisc}
\end{eqnarray}
where $\boldsymbol{\rho}=(\rho_{i,j})$ and $\mathbf{W}^x, \mathbf W^y$ are independent $IJ$-dimensional Brownian motions, $\mathbf{J}$ is an $IJ \times IJ$ diagonal matrix of $\jacdat$, and $\boldsymbol{\nabla}_h^+$ is $IJ \times 2IJ$ matrix that represents the discrete divergence and $\boldsymbol{\nabla}_h^-$ is the $2IJ \times IJ$ matrix that represents the discrete gradient.
$\mathrm{Diag}(\boldsymbol{\alpha})$ is a $2IJ \times 2IJ$ diagonal matrix formed from $2\times 2$ blocks of the form 
\[
\begin{pmatrix}
    \alpha_{i,j} & 0 \\ 0 & \alpha_{i,j}
\end{pmatrix} .
\]
Finally, $\mathbf{G}^{-1}$ and $\mathbf{G}^{-1/2}$ are block diagonal $2IJ \times 2IJ$ matrices where the blocks are the $2\times 2$ matrices $G_{i,j}^{-1}$ and $G_{i,j}^{-1/2}$, respectively.

To derive the fluctuation-dissipation relation, we now consider the system \eqref{num_rho} linearized around the time-stationary mean-field limit $\bar{\rho} = \frac{1}{\nu(D)}$, which can be written as
\begin{equation}
     \mathd\boldsymbol{Z}= \, \mathbf{J}^{-1} \mathbf{L}_h \boldsymbol{Z}\, \mathd t 
    +\sqrt{\frac{2 \bar{\rho}}{N \Delta x \Delta y}} \, \mathbf{J}^{-1}
    \mathbf{K}_h
    \begin{pmatrix}
        \mathd\mathbf{W}^x \\[.1cm]
        \mathd\mathbf{W}^y
    \end{pmatrix},
    \label{eq:fv_matdiscequil}
\end{equation}
where
\[
\mathbf{L}_h = \boldsymbol{\nabla}_h^+ (\mathbf{J} \mathbf{G}^{-1}  \boldsymbol{\nabla}_h^-)
\qquad \mathrm{and} \qquad  \mathbf{K}_h = \boldsymbol{\nabla}_h^+  \,(\mathbf{J}^{1/2}\, \mathbf{G}^{-1/2}) \;\;.
\]
The stationary covariance of $\boldsymbol{Z}$, denoted by $C_{\boldsymbol{Z}}$, is then given by
\begin{equation}\label{eq_Z}
\mathbf{J}^{-1} \mathbf{L}_h C_{\boldsymbol{Z}} + C_{\boldsymbol{Z}} \mathbf{L}_h^* \mathbf{J}^{-1} = 
- \frac{2 \bar{\rho}}{N \Delta x \Delta y} \, \mathbf{J}^{-1}
    \mathbf{K}_h \mathbf{K}_h^* \mathbf{J}^{-1}.
\end{equation}

Since by construction, $(\boldsymbol{\nabla}_h^+)^*=-\boldsymbol{\nabla}_h^-$  and the above operators are represented by matrices, we may use associativity of matrix multiplication to obtain, 
\begin{align}
    \boldsymbol{K}_h   \boldsymbol{K}_h^* &=\boldsymbol{\nabla}_h^+( \,\mathbf{J}^{1/2}\, \mathbf{G}^{-1/2} )(\mathbf{G}^{-1/2})^T \,\mathbf{J}^{1/2}\,  (\boldsymbol{\nabla}_h^+ )^*\\
    &= -\boldsymbol{\nabla}_h^+( \,\mathbf{J} \,  \mathbf{G}^{-1} \boldsymbol{\nabla}_h^-) = -\boldsymbol{L}_h.
\end{align}
Hence, the equation \eqref{eq_Z} becomes
\begin{equation*}
    \mathbf{J}^{-1} \mathbf{L}_h C_{\boldsymbol{Z}} + C_{\boldsymbol{Z}} \mathbf{L}_h^* \mathbf{J}^{-1} = 
 \frac{2 \bar{\rho}}{N \Delta x \Delta y} \, \mathbf{J}^{-1}
    \mathbf{L}_h  \mathbf{J}^{-1} \;\;\;,
\end{equation*}
which then gives

\begin{equation}
    C_{\boldsymbol{Z}} = \frac{\bar{\rho}}{N \Delta x \Delta y} \, \mathbf{J}^{-1}.
\label{eq:disc_FDR}
\end{equation}
Equation~(\ref{eq:disc_FDR}) shows the fluctuation dissipation relation for the spatially discrete system in coordinate space.  This result is consistent with the fluctuation dissipation result \eqref{cFDR} discussed  for the linearized SPDE, which implies that the numerical method will generate the correct equilibrium statistics.

We now consider the temporal discretization of Equation~(\ref{eq:fv_semidisc}), assuming the metric tensor and the external potential do not depend on time.  For simplicity we will only consider an Euler--Maruyama discretization to avoid complications associated with higher order temporal discretization of stochastic differential equations with multiplicative noise.  This then gives the final discrete form
\begin{align}
    \rho_{i,j}^{n+1}  = \rho_{ij}^n \ + \ & \Delta t \Biggl[  \frac{1}{\sqrt{|G|}_{i,j}}\nabla_h^+ \cdot (\sqrt{|G|}_{i,j}G_{i,j}^{-1}\nabla_h^- \rho_{i,j}^n) + \frac{1}{\sqrt{|G|}_{i,j}}\nabla_h^+ \cdot ( \sqrt{|G|}_{i,j} \rho_{i,j}^n G_{i,j}^{-1}\nabla_h^- V_{i,j}) \nonumber \\
    \ +\ & \sqrt{\frac{2}{N}} \frac{1}{\sqrt{|G|}_{i,j}} \nabla_h^+ \cdot \left(\sqrt{\rho}_{i,j}^n |G|_{i,j}^{1/4}G_{i,j}^{-1/2} \frac{1}{\sqrt{ \Delta t \Delta x \Delta y}} \begin{pmatrix}
        Z_{i,j}^x \\[.1cm]
        Z_{i,j}^y
    \end{pmatrix} \right) \Biggr],
    \label{eq:fv_disc}
\end{align}
where the $Z_{i,j}^{x,y}$ are centered normally distributed random numbers with
\[
\mathbb E[ Z_{i,j}^{\alpha}, Z_{i',j'}^{\alpha'}] = \delta_{i,i'} \; \delta_{j,j'} \; \delta_{\alpha,\alpha'},
\]
for Kronecker deltas $\delta$.

\section{Numerical experiments}\label{sec:experiments}

In this section we present several computational examples demonstrating the behavior of the surface Dean--Kawasaki equation.

\subsection{Equilibrium test}

The first numerical example investigates the equilibrium statistical properties of the discretized system. 
For this case we consider the surface defined by
\[
H(x,y) = a \sin(x) \sin(y). 
\]
The explicit formulae for the metric terms are given in  \ref{AppendixA}.

We discretize the system on $[0,2\pi)^2$ with $a=3$ on a $32 \times 32$ grid
In Figure \ref{fig:surface_1}, we show the resulting surface $H(x,y)$ colored by the local surface density.
\begin{figure}[h]
    \centering
    \includegraphics[width=0.7\textwidth]{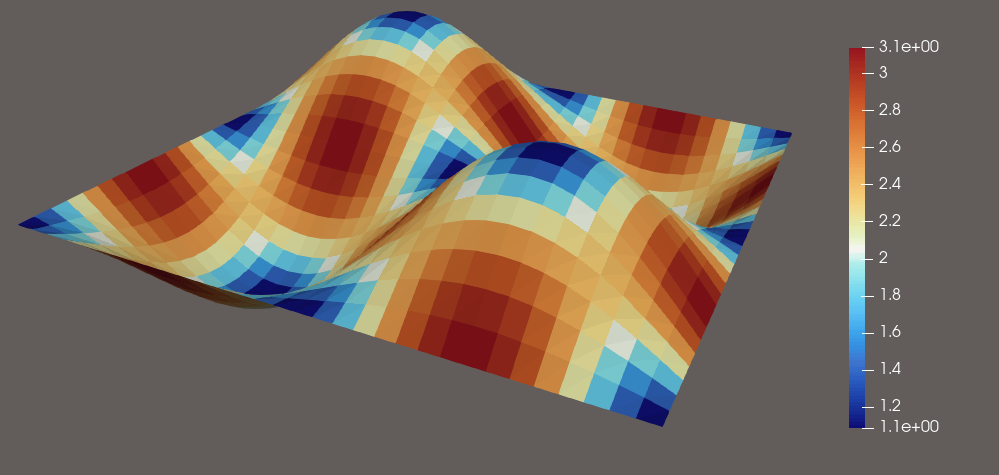}
    \caption{Image showing $H(x,y)$ colored by the local surface density $\sqrt{|G|}$.}
    \label{fig:surface_1}
\end{figure}
We
set the time step to $\Delta t = 1.506 \times 10^{-4}$, which is $1.5625\times 10^{-2}$ time the maximum stable time step for the discretization.  The small time step minimizes errors associated with the Euler--Maruyama scheme. 

We set the total number of particles to $N = 10240$ and initialize with constant value
\[
\rho = \frac{1}{\int_\mathcal{D} \jacdat \mathd x},
\]
which corresponds to the steady state solution of the system without noise.
Here the integral is computed by summing $\Delta x \Delta y \jacdat$ over the grid. 
The system is evolved for $10^5$ time steps to reach equilibrium. 
We then run for an additional $10^6$ steps, gathering statistics every 10 steps.  For comparison purposes, we also integrate the particle system given by \eqref{eq:BMonSurface} for the same length of time using the same time step, again with Euler--Maruyama.  For the particle integration, we initialize the system with 10 particles per cell so that both systems have the same number of particles.  We note that the particle initialization is not the same as the SPDE initialization; however, these differences have disappeared before we begin collecting statistics. 

We consider two different representations of the data, one in terms of $\rho$ and the other in terms of the number of particles in each cell, which we denote as $N_{i,j}$.  To map between these representations, we note that
\begin{equation}
N_{i,j} = N \, \rho_{i,j} \;  \jacdat \,\Delta x \Delta y
\label{eq:map_mu_n}.
\end{equation}
We note that while the equilibrium mean of $\rho$ is constant, the mean number of particles per cell reflects the local surface density given by $\sqrt{|G|}$; i.e., there are more particles in regions with higher local surface density. 

In Figure \ref{fig:mu_mean} we plot the mean values obtained from finite volume and particles simulations.  The mean values of $\rho$ on the grid are converging to the theoretical value 
\[\frac{1}{\int_\mathcal{D} \jacdat \mathd x} \;\;
\]

The data is plotted in terms of number of particles per cell in Figure \ref{fig:num_mean}, where in this case the theoretical mean number of particles per cell is given by 
\[
\frac{N \;  \jacdat \,\Delta x \Delta y }{ {\int_\mathcal{D} \jacdat \mathd x}} \;\; .
\]

The mean values of $N_{i,j}$ reflect the underlying structure of the surface as expected. More precisely, the average number of particles per cell is proportional to the area of the surface in that cell.  As seen in Figure \ref{fig:num_mean} both the finite volume method and the particle method recover the correct mean distribution.

\begin{figure}[h]
    \centering
    \includegraphics[width=0.287\textwidth]{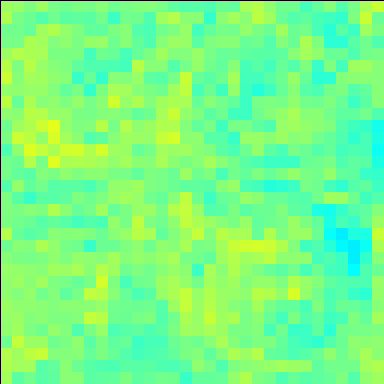}
    \includegraphics[width=0.287\textwidth]{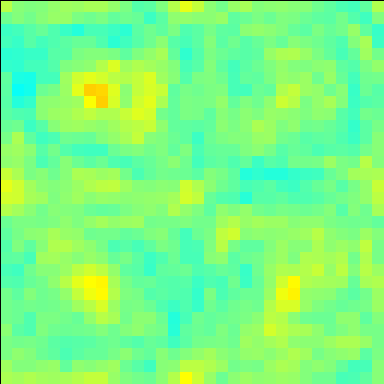}
    \includegraphics[width=0.287\textwidth]{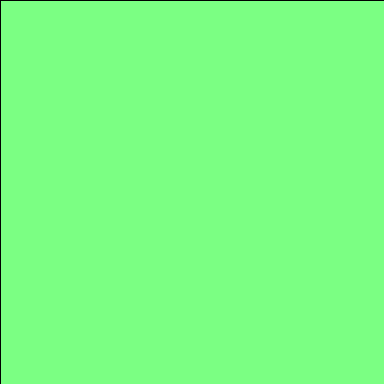}
    \includegraphics[width=0.105\textwidth]{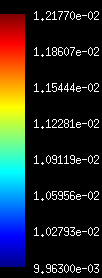}
    \caption{Mean of $\rho$.  The left panel is the finite volume scheme, the middle panel is a Brownian particles simulation and the right panel is the theoretical result.  Note that we have set the range to correspond to the theoretical mean $\pm 10\%$.
    }
    \label{fig:mu_mean}
\end{figure}

\begin{figure}[h]
    \centering
    \includegraphics[width=0.287\textwidth]{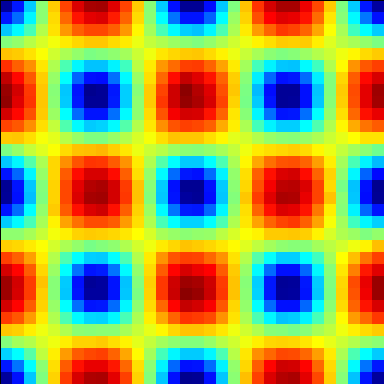}
    \includegraphics[width=0.287\textwidth]{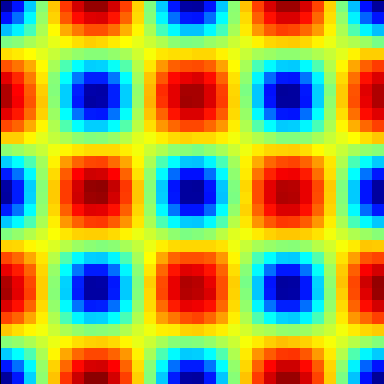}
    \includegraphics[width=0.287\textwidth]{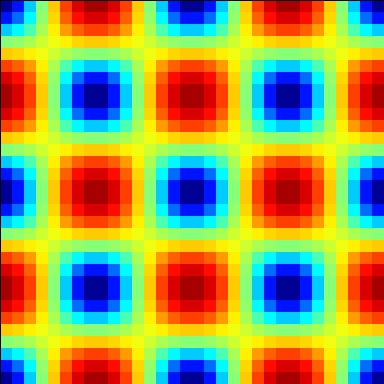}
    \includegraphics[width=0.105\textwidth]{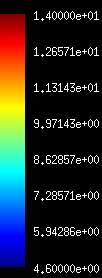}
    \caption{Mean number of particles in each cell.  The left panel is the finite volume scheme, the middle panel is a Brownian particles simulation and the right panel is the theoretical result.
    }
    \label{fig:num_mean}
\end{figure}

Next we consider the variance of the equilibrium distributions computed from the simulations.  Although the mean
value of $\rho$ is constant over the domain, the variance of $\rho$ is not.  In particular, as discussed above in
Equation~(\ref{eq:disc_FDR}),
\[
\mathrm{Var}(\rho_{i,j}) = \frac{1}{N A_S \jacdat_{i,j} \,\Delta x \Delta y }
\]
where $A_S = {\int_{\mathcal{D}} \jacdat \mathd x}$ is the surface area. This reflects the intuition that $\rho$  has higher variability in cells where the local surface area is smaller.  The variances of $\rho$ obtained from the simulations are show in Figure \ref{fig:mu_var}.  Again, both the finite volume and particle data are in good agreement with the theoretical values.

\begin{figure}[h]
    \centering
    \includegraphics[width=0.287\textwidth]{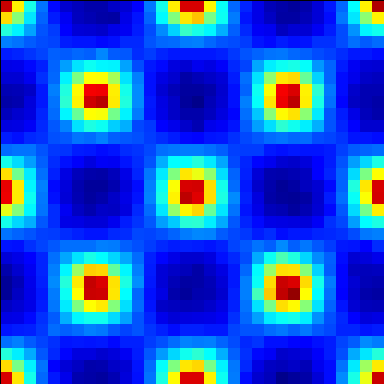}
    \includegraphics[width=0.287\textwidth]{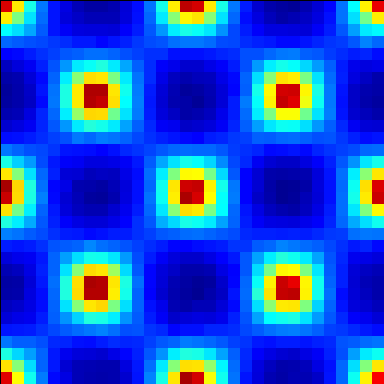}
    \includegraphics[width=0.287\textwidth]{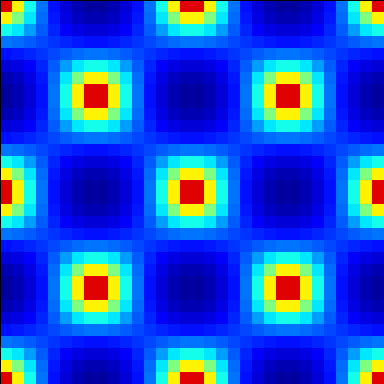}
    \includegraphics[width=0.105\textwidth]{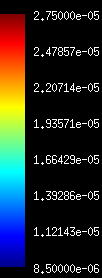}
    \caption{Variance of $\rho$.  The left panel is the finite volume scheme, the middle panel is a Brownian particles simulation and the right panel is the theoretical result.
    }
    \label{fig:mu_var}
\end{figure}

The variance of $N_{i,j}$ is shown in Figure \ref{fig:num_var}.  The expected variance of $N_{i,j}$ is given by
\[
\mathrm{Var}(N_{i,j}) = \frac{N  \jacdat_{i,j} \,\Delta x \Delta y } {A_S}.
\]
When expressed in terms of number density there is more variability in regions with higher local surface area, which is the opposite of what was observed for the variance of $\rho$.  Again, the figures show excellent agreement between the two  simulations and the theoretical result. We also note that in the simulations $\mathrm{Var}(N_{i,j}) \approx \overline{N}_{i,j}$.  This reflects the underlying Poisson character of the particle equilibrium distribution.  In particular, we expect the distribution of particles within each cell to be Poisson distributed with a mean and variance proportional to the local surface area in that cell.

\begin{figure}[h]
    \centering
    \includegraphics[width=0.287\textwidth]{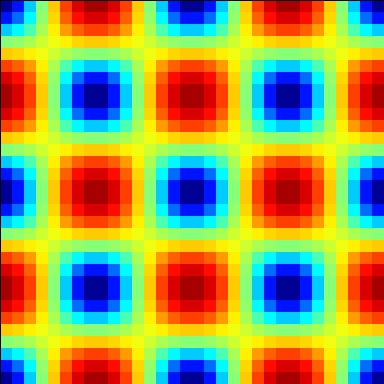}
    \includegraphics[width=0.287\textwidth]{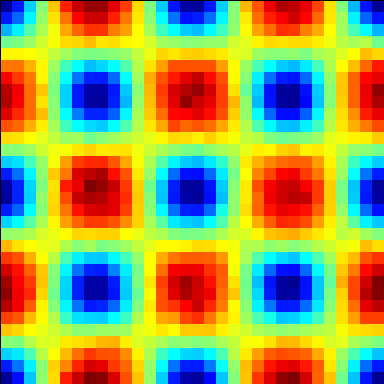}
    \includegraphics[width=0.287\textwidth]{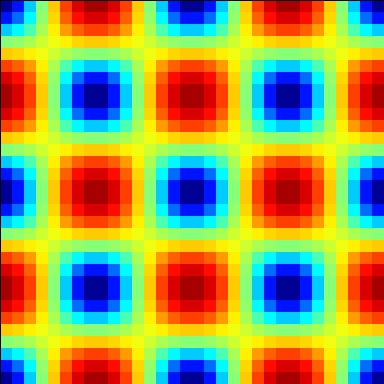}
    \includegraphics[width=0.105\textwidth]{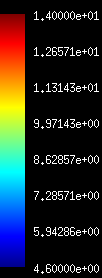}
    \caption{The variance of the number of particles in each cell.  The left panel is the finite volume scheme, the middle panel is a Brownian particles simulation and the right panel is the theoretical result.
    }
    \label{fig:num_var}
\end{figure}

\subsection{Transient simulations for a fixed surface}\label{sec:trans}

In this section we illustrate the transient behavior of the surface Dean--Kawasaki equation.
For the example presented here, the surface is given by
\[
H(x,y) = 4 \sin^2(x) \sin^2(y).
\]
The explicit formulae for the metric terms are given in \ref{AppendixB}.

We discretize the system on $[0,2\pi)\times[0,2\pi)$ on a $64 \times 64$ grid and again we
set the time step to $\Delta t = 1.506 \times 10^{-4}$ which is $1.5625\times 10^{-2}$ time the maximum stable time step for the discretization.  The small time step minimizes errors associated with the Euler--Maruyama scheme. 

We set the total number of particles to $N = 100000$ and initialize

\[
\rho_{i,j}= \Biggl\{ 
\begin{matrix*}[l]
   0.801421 & \mathrm{if} \;\;\; \sqrt{(x_{i,j}-\pi)^2+(y_{i,j}-\pi)^2} < 0.2 \pi
    \\[.1cm]
    0 & \mathrm{otherwise} 
\end{matrix*},
\]
which satisfies $\int_{\mathcal{D}} \rho(x) \mathd x = 1$.
The surface consists of four peaks of height 4 at 
\[(\pi\pm \pi/2, \pi\pm \pi/2)\]
and the initial condition is nonzero in a small circle between the four peaks.

The simulation was run for 50,000 time steps. Figure \ref{fig:mu_evolve} shows a time sequence of the solution plotted on the surface.
In the image, we have reduced height of the peaks so that $\rho$ is not obscured.  Also, we have set the range of $\rho$ to be the same for all the images.  The peak of $\rho$ is initially around 0.8 but decays rapidly over time.  Consequently the central red region of the images at early time are over saturated.
Eventually, $\rho$ would evolve to reach a constant value over the entire domain. However, at early times $\rho$ diffuses more slowly into steep regions where the local area is large.  This can be seen in the images where $\rho$ is 
rapidly diffusing through the valleys and more slowly diffusing up the steep parts of the surface.
\begin{figure}[h]
    \centering
    t = 0.47 \hspace{2in} t = 0.94 \\
    \vspace{.05in}
    \includegraphics[width=0.4\textwidth]{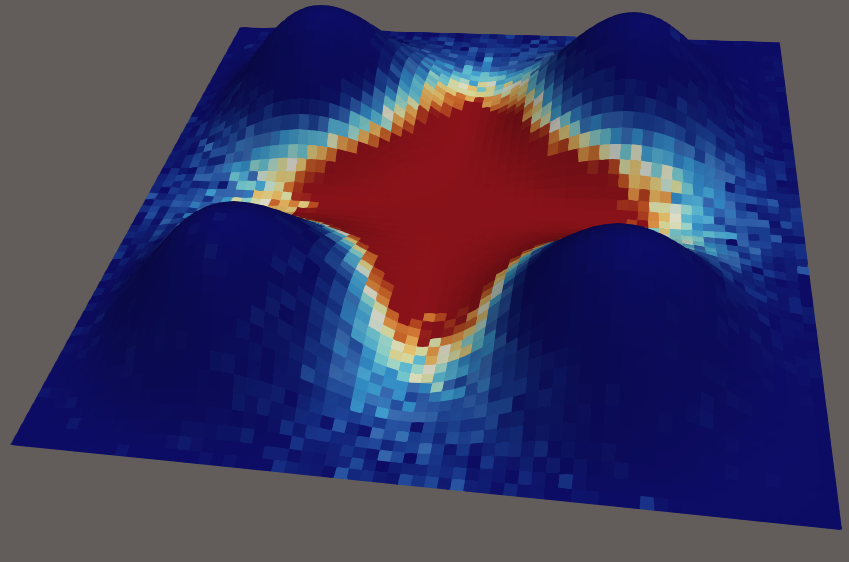}
    \includegraphics[width=0.4\textwidth]{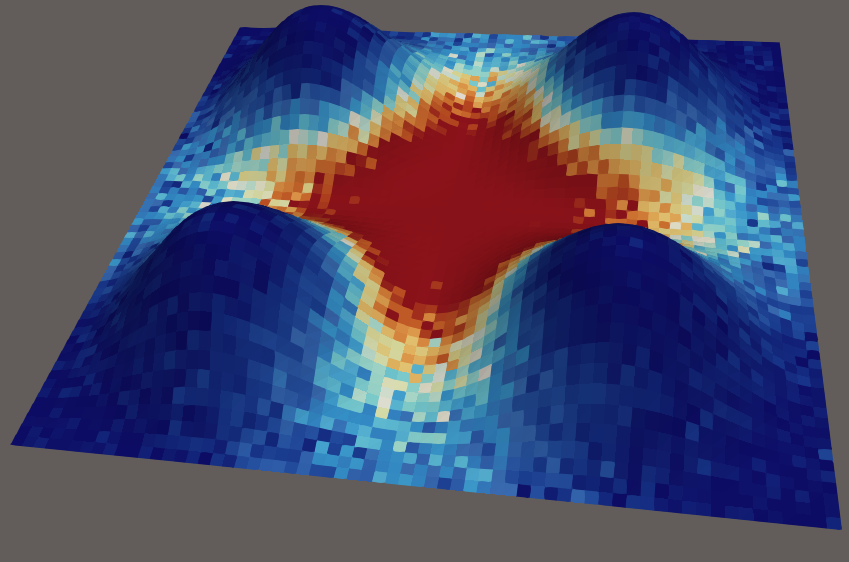} 
    
    t = 1.41 \hspace{2in} t = 1.88 \\
      \vspace{.05in}
 \includegraphics[width=0.4\textwidth]{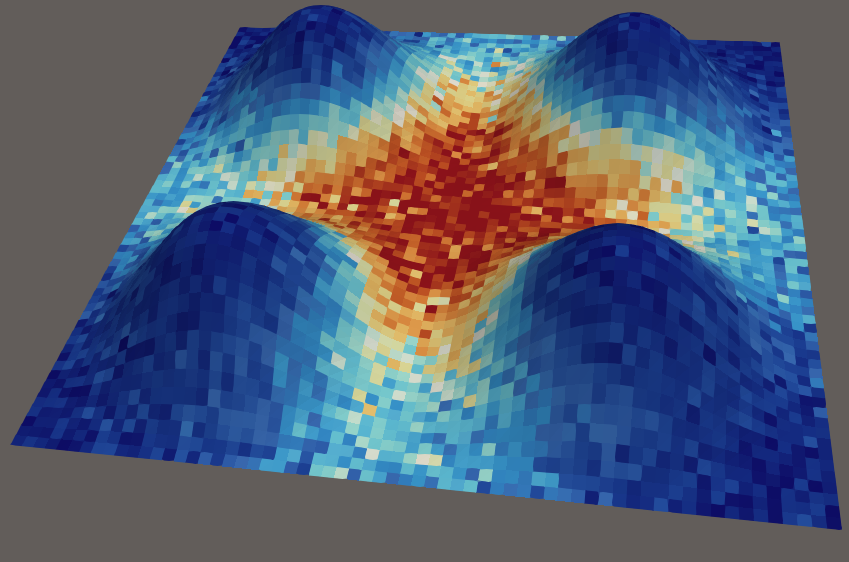}
 \includegraphics[width=0.4\textwidth]{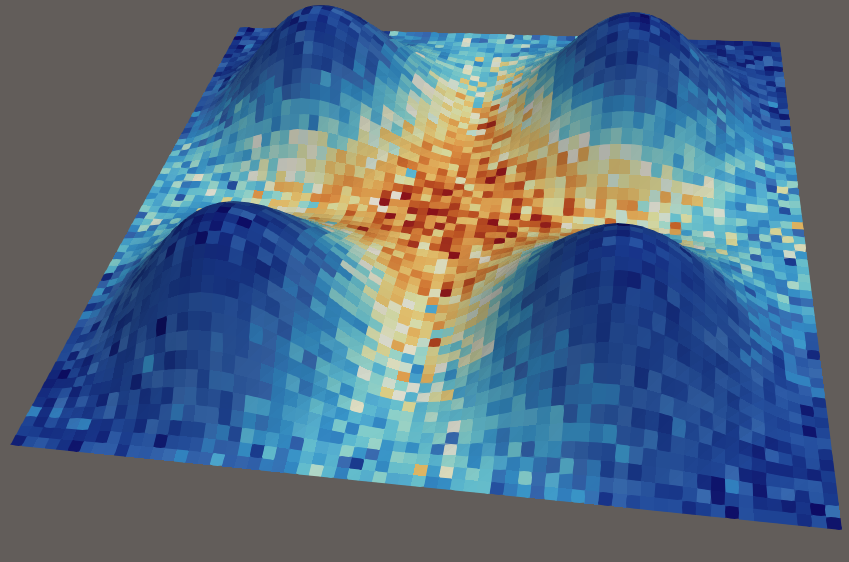}
    \caption{Time evolution of $\rho$. 
    The range of $\rho$ in the image is restricted to $[0,0.4]$ so the values are clipped particularly at early times.
    The peak values of $\rho$, left to right, are 0.166, 0.0786, 0.0542 and 0.0457. The height of the peaks on the surface have been scaled by 0.3 so that $\rho$ is not obscured.
    }
    \label{fig:mu_evolve}
\end{figure}

In Figure \ref{fig:num_evolve} we present the results of the same simulation expressed in terms of number of particles per cell, $N_{i,j}$. As we did before, we have set the range of $N_{i,j}$ to be the same for all images, which results in clipping of the values particularly at early times.  
The key observation in these images is that the diffusive process is slowed in regions where the surface is steep, and the number of particles in those regions is quite high. 
This reflects the notion that we expect more particles to be present in cells with higher local surface area.
Consequently, as the particles diffuse through the domain we observe high number of particles in the steep regions of the peaks and the center of the domain.
\begin{figure}[h]
    \centering
        t = 0.47 \hspace{2in} t = 0.94 \\
    \vspace{.05in}
    \includegraphics[width=0.4\textwidth]{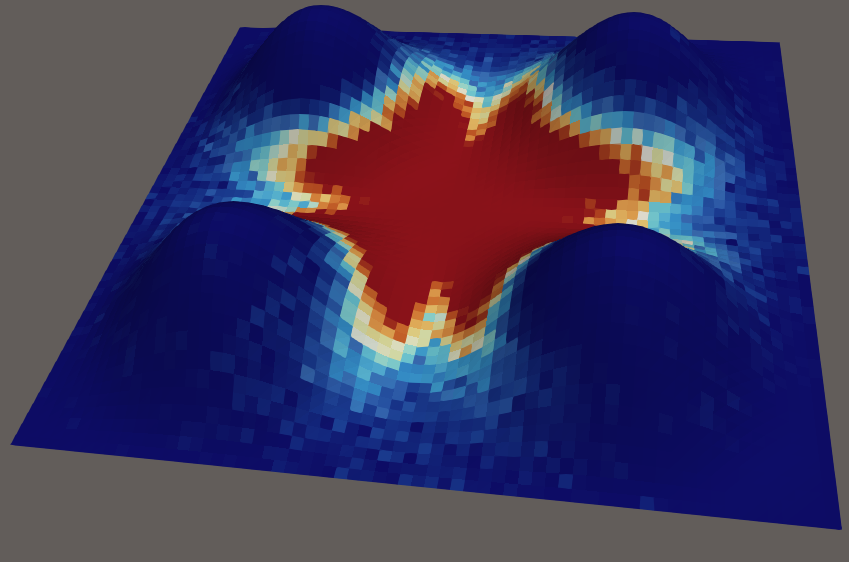}
    \includegraphics[width=0.4\textwidth]{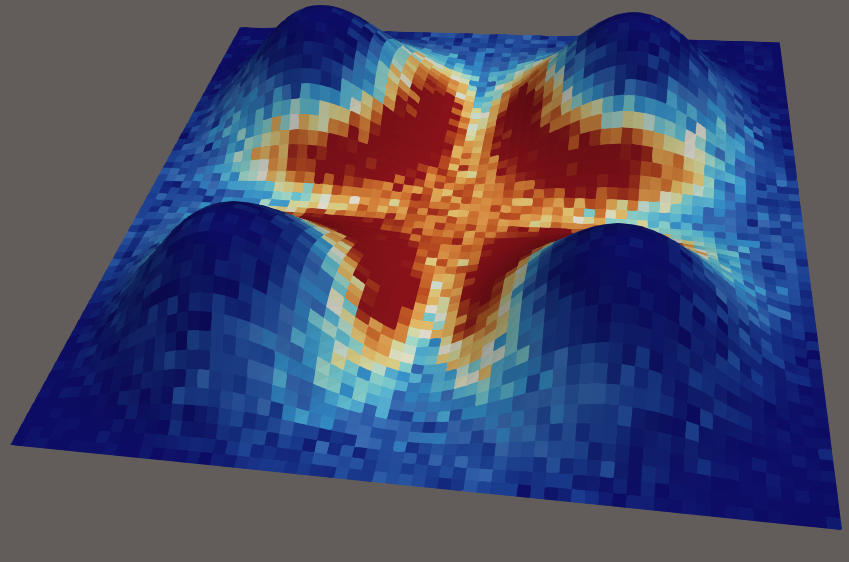}
        
    t = 1.41 \hspace{2in} t = 1.88 \\
      \vspace{.05in}
    \includegraphics[width=0.4\textwidth]{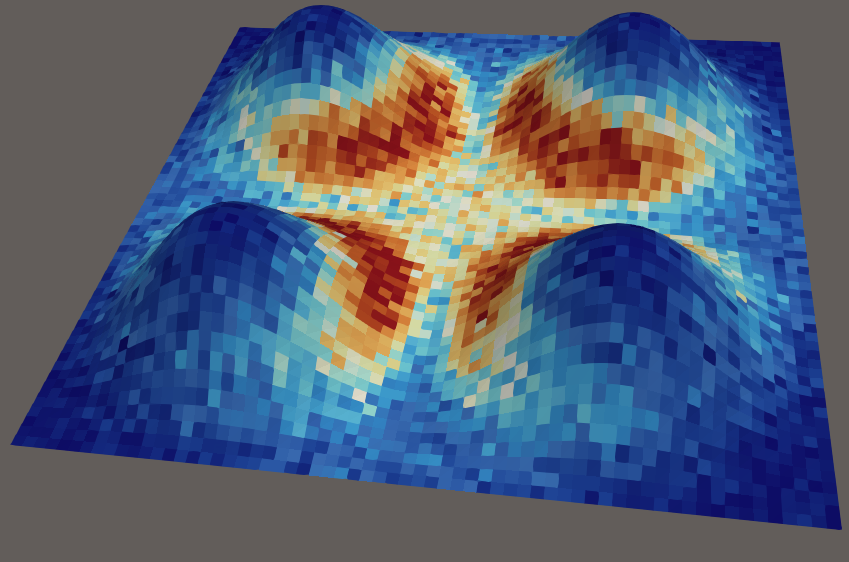}
    \includegraphics[width=0.4\textwidth]{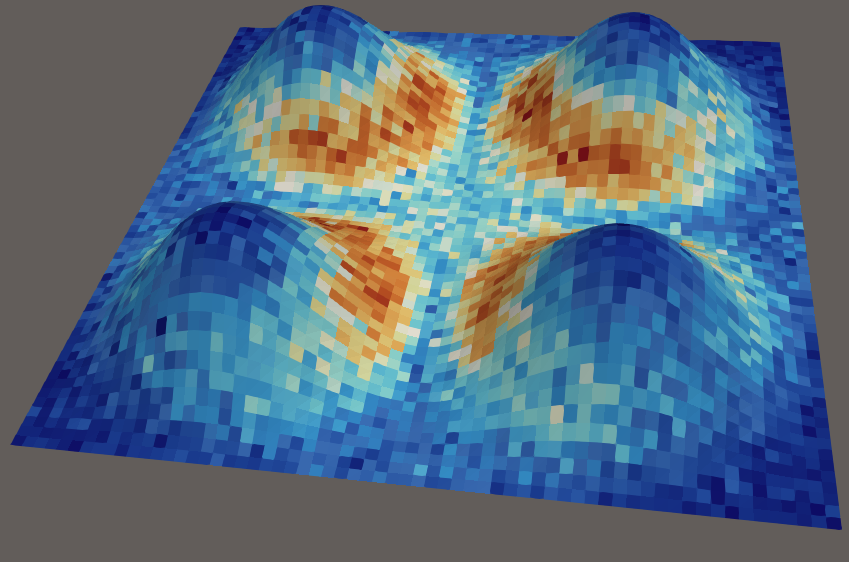}
    \caption{Time evolution of $N_{i,j}$. 
    The range of $N_{i,j}$ in the image is restricted to $[0,80]$ so the values are clipped particularly at early times.
    The peak values of $N_{i,j}$, left to right are 177, 135, 103 and 81. As in Figure \ref{fig:mu_evolve} height of the peaks on the surface have been scaled by 0.3 so that the solution is not obscured.
    }
    \label{fig:num_evolve}
\end{figure}

\subsection{Effect of an external potential}

In this section, we present results for same surface as was used in Section \ref{sec:trans}; however, now we introduce an external potential.  Specifically, we set
\[
V(x,y) = 5 \sin^2(x) \sin^2(y).
\]
Although the potential is of the same form as the surface, the effect of the potential on the dynamics is quite different.  From the perspective of the original particles dynamics, a steep region of the surface slows the particles motion relative to the $(x,y)$ coordinate system, whereas a steep region in the potential biases the particle motion to move downhill to a lower value of the potential.

This type of behavior is illustrated in Figure \ref{fig:mu_evolve_wp}.  The simulation parameters and the initial conditions are the same as in Section \ref{sec:trans}. For comparison purposes, we have used the same scaling here as in Figure \ref{fig:mu_evolve}.  As seen in the images, $\rho$ is more localized in the valleys of the surface, which also corresponds to the minimum of the potential surface.  The presence of the potential suppresses $\rho$ diffusing up the peaks and leads to higher peak  values of the solution.  The focusing effect of the potential also results in higher values of $\rho$ reaching the boundary and one begins to see the solution diffusing around the outer portion of the peaks.
\begin{figure}[h]
    \centering
        t = 0.47 \hspace{2in} t = 0.94 \\
    \vspace{.05in}
    \includegraphics[width=0.4\textwidth]{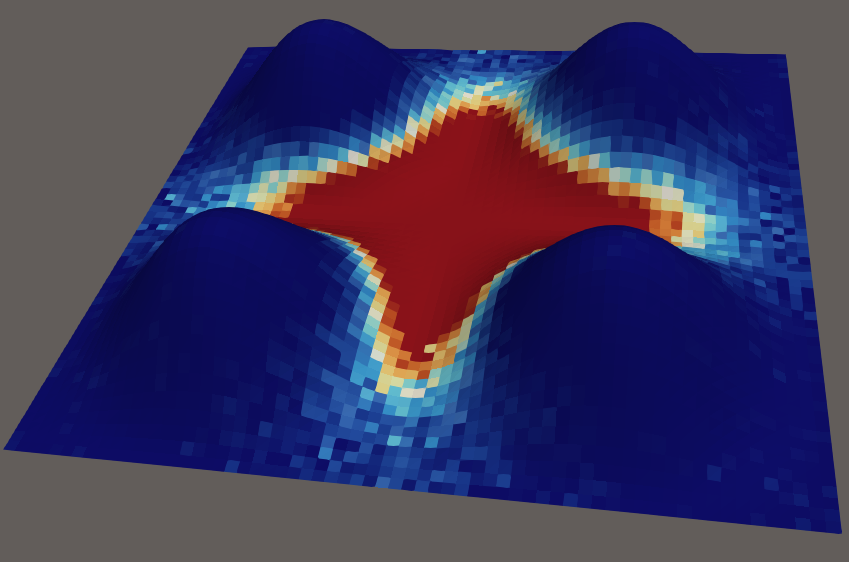}
    \includegraphics[width=0.4\textwidth]{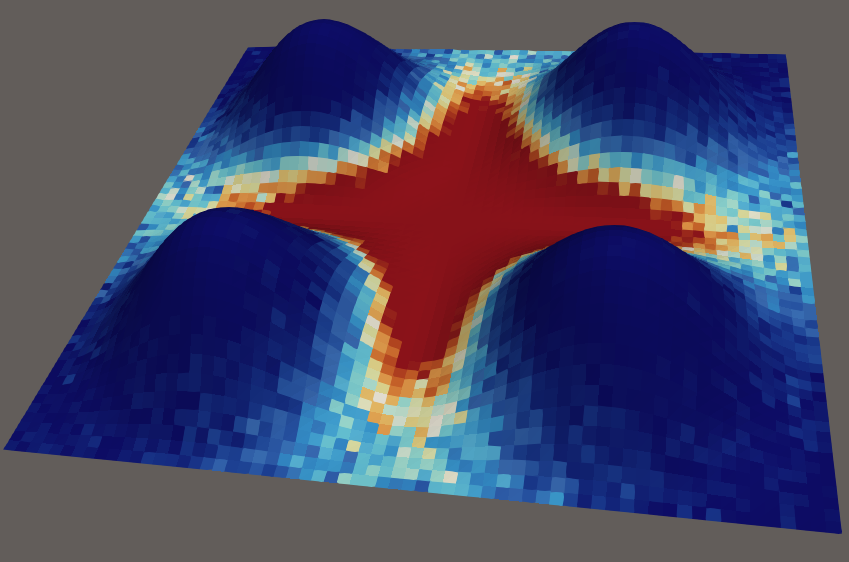}
        
    t = 1.41 \hspace{2in} t = 1.88 \\
      \vspace{.05in}
    \includegraphics[width=0.4\textwidth]{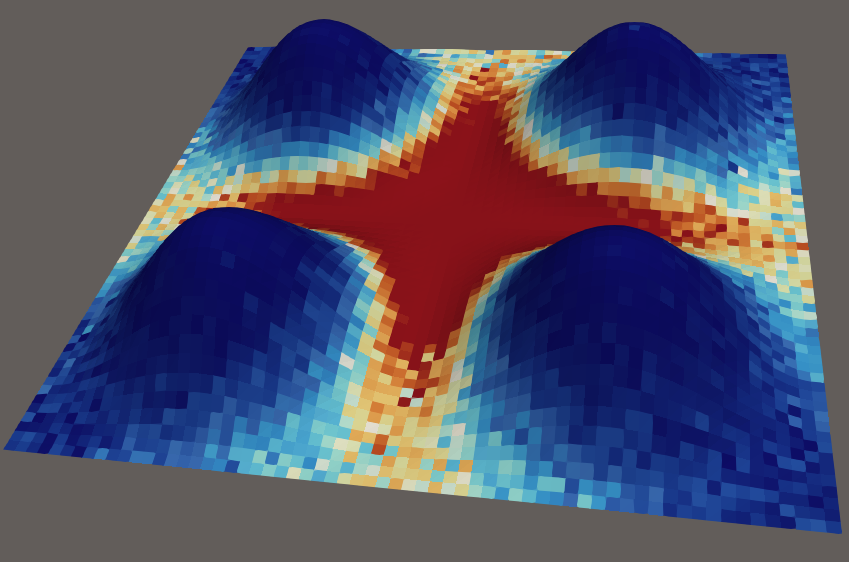}
    \includegraphics[width=0.4\textwidth]{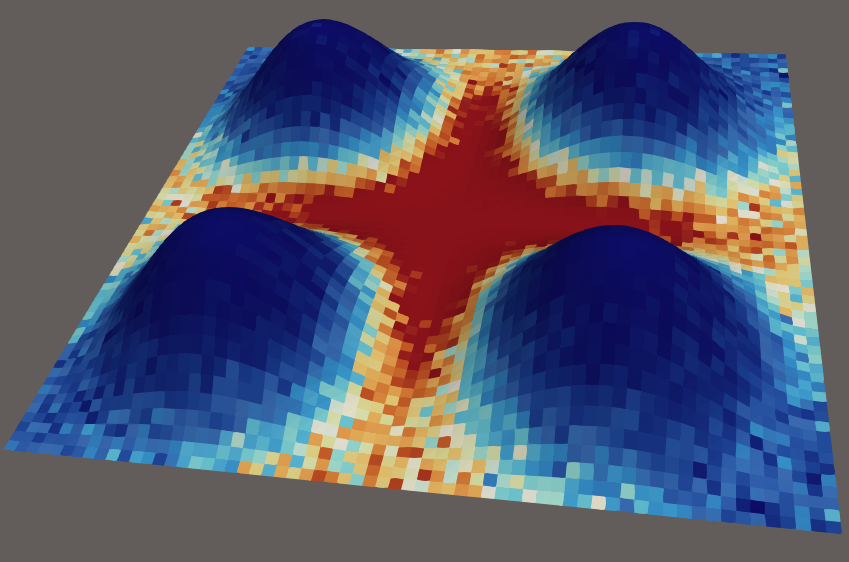}
    \caption{Time evolution of $\rho$ with an external potential. 
    The range of $\rho$ in the image is restricted to $[0,0.4]$ so the values are clipped particularly at early times.
    The peak values of $\rho$, left to right are 0.203, 0.125, 0.0957 and 0.0856. As before, the height of the peaks on the surface have been scaled by 0.3 so that $\rho$ is not obscured.
    }
    \label{fig:mu_evolve_wp}
\end{figure}

Figure \ref{fig:num_evolve_wp} shows the results of the simulation with an external potential expressed in terms of number of particles per cell.  As with $\rho$ the potential leads to a concentration of $N_{i,j}$ in the lower potential regions of the domain.  The most striking feature of images is that the potential suppresses the high concentration of particles on the steep inner part of the surface.  The two-pronged structure where the solution is diffusing down the channel between the peaks reflects the tradeoff between the tendency to have more particles in regions where the surface is steeper and the impact of the potential on biasing the particles toward the low potential region in the center of the channel.
\begin{figure}[h]
    \centering
        t = 0.47 \hspace{2in} t = 0.94 \\
    \vspace{.05in}
    \includegraphics[width=0.4\textwidth]{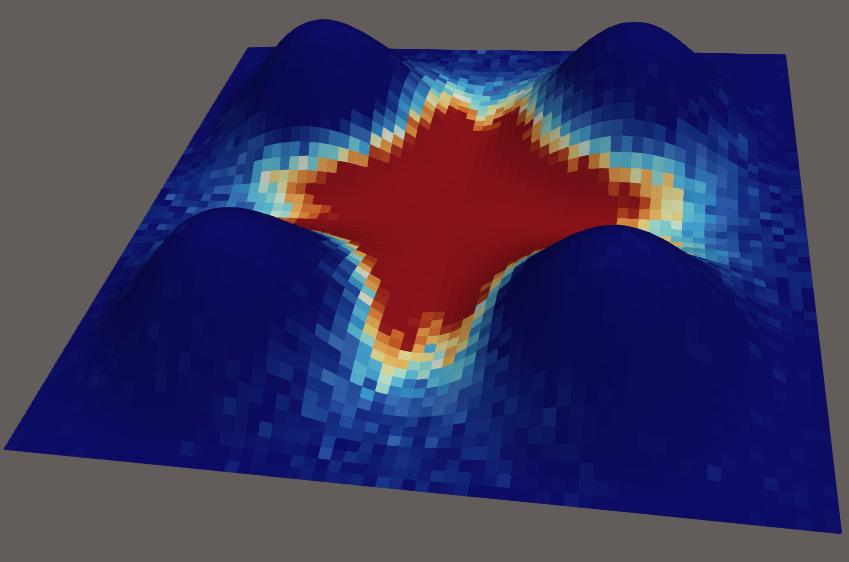}
    \includegraphics[width=0.4\textwidth]{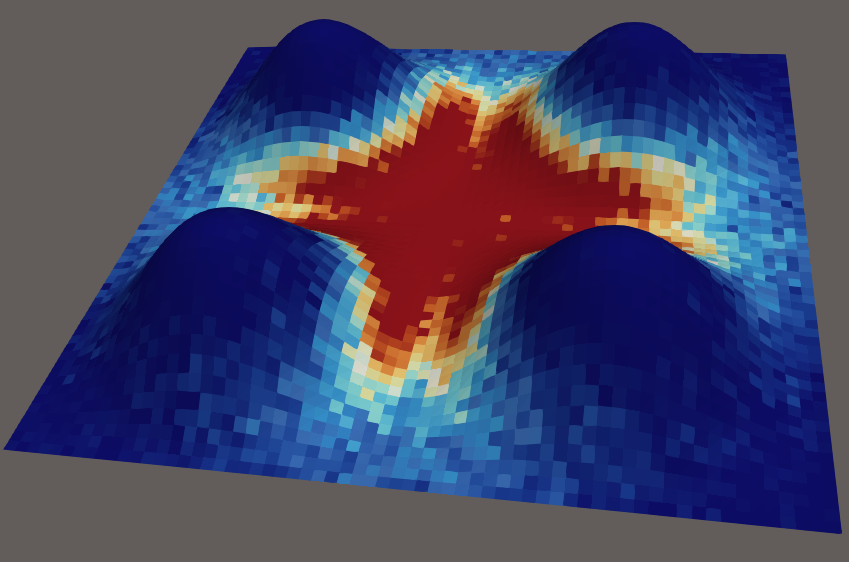}
        
    t = 1.41 \hspace{2in} t = 1.88 \\
      \vspace{.05in}
    \includegraphics[width=0.4\textwidth]{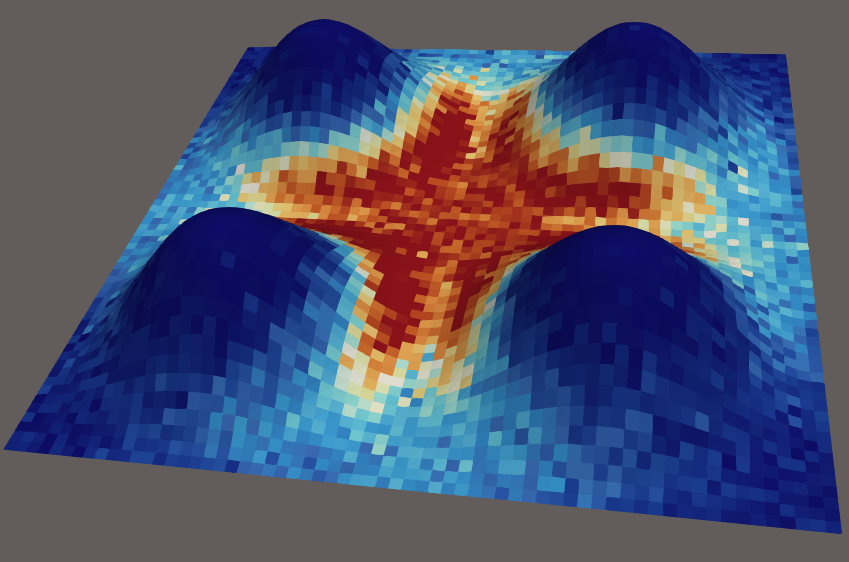}
    \includegraphics[width=0.4\textwidth]{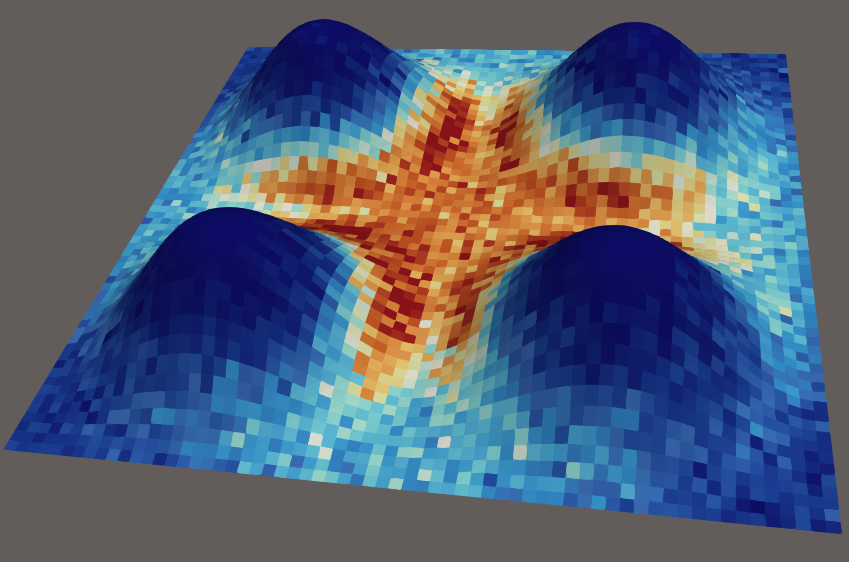}
    \caption{Time evolution of $N_{i,j}$ with an external potential. 
    The range of $N_{i,j}$ in the image is restricted to $[0,80]$ so the values are clipped particularly at early times.
    The peak values of $N_{i,j}$, left to right are 196, 144, 123 and 109. As before, the height of the peaks on the surface have been scaled by 0.3 so that the solution is not obscured.
    }
    \label{fig:num_evolve_wp}
\end{figure}

\subsection{Simulation of particles on a deforming surface}

In the last example we consider the dynamics of particles on a moving surface.
For this example, the time-dependent surface is given by
\[
H(x,y,t) = 4 \left(\alpha(t) \sin^2(x) \sin^2(y) + (1-\alpha(t)) \cos^2(x) \cos^2(y) \right)
\]
where
\[
\alpha(t) = \left\{ 
\begin{array}{ll}
\cos^2 (\pi t / 2 t_{sc}) & t \leq t_{sc} \\
0  & t > t_{sc}
\end{array} \right .
\]
For the case considered here, we set $t_{sc} = 0.3$ so that the initial surface changes faster than the particle diffusion.  Once the simulation has reached $t = t_{sc}$ the surface no longer changes.

We discretize the system on $[0,2 \pi]^2$ with $N=100,000$ particles on a $64 \times 64$ grid. We use a time step of $3.765 \times 10^{-5}$ to minimize errors associated with the Euler-Maruyama temporal discretization. We simulate the dynamics for 15,000 steps with initial condition
\[
\rho = \frac{1}{\int_\mathcal{D} \sqrt{|G{x,0}|} \; dx} \;\; ,
\]
which corresponds to the mean distribution corresponding to the surface at time $t=0$.

In Figure \ref{fig:phi_move_evolve} we show a time sequence of the evolution
of $\rho$ as the surface deforms. At $t = 0.15$ we see that the change of the surface has lead to a concentration of $\rho$ in regions where the surface was initially steeper.  This trend is intensified at $t=0.3$ as the new peaks emerge.  After $t=0.3$ the surface no longer evolves and we see the beginning of the relaxation to a new equilibrium with mean of
\[
\rho = \frac{1}{\int_\mathcal{D} \sqrt{|G{x,t_{sc}}|} \; dx} \;\; .
\]
In Figure \ref{fig:num_move_evolve} we present the results of the simulation in terms of $N_{i,j}$.  At $t=0.00$ we see large numbers of particles clustered on the sides of the bumps where the surface is steepest.  As the surface evolves the particles diffuse away from those regions of initially high concentration and begin to cluster in steep regions of the new bumps that emerge from the surface evolution.
\begin{figure}[h]
    \centering
    t = 0.00 \hspace{2in} t = 0.15 \\
    \vspace{.05in}
    \includegraphics[width=0.4\textwidth]{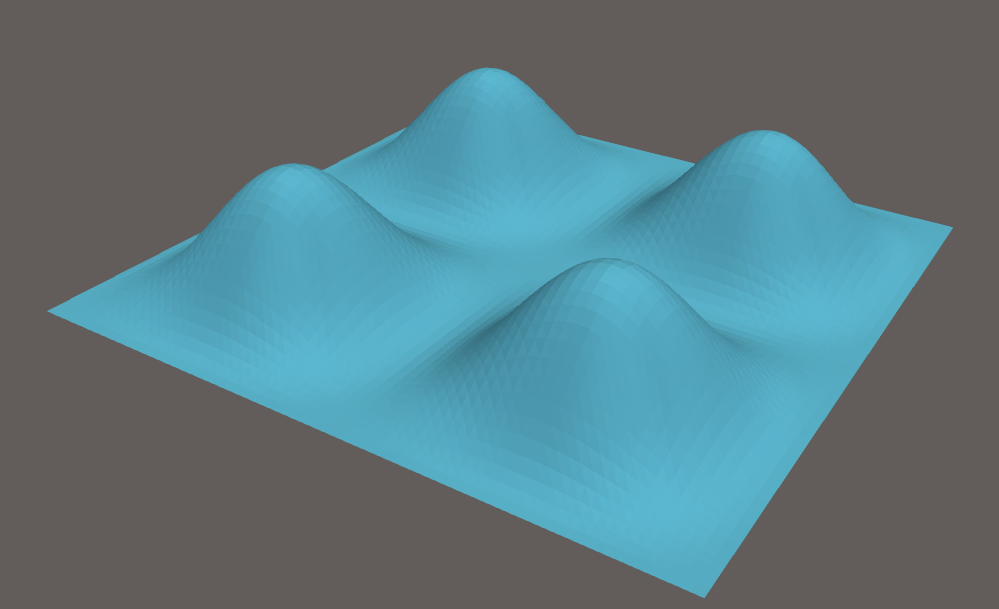}
    \includegraphics[width=0.4\textwidth]{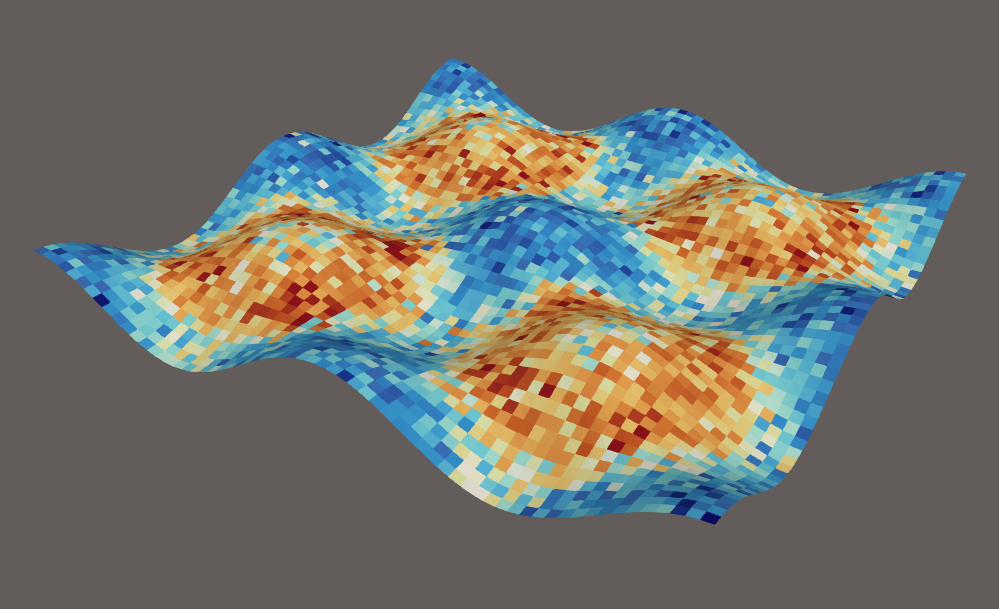} 
    
    t = 0.30 \hspace{2in} t = 0.56 \\
      \vspace{.05in}
 \includegraphics[width=0.4\textwidth]{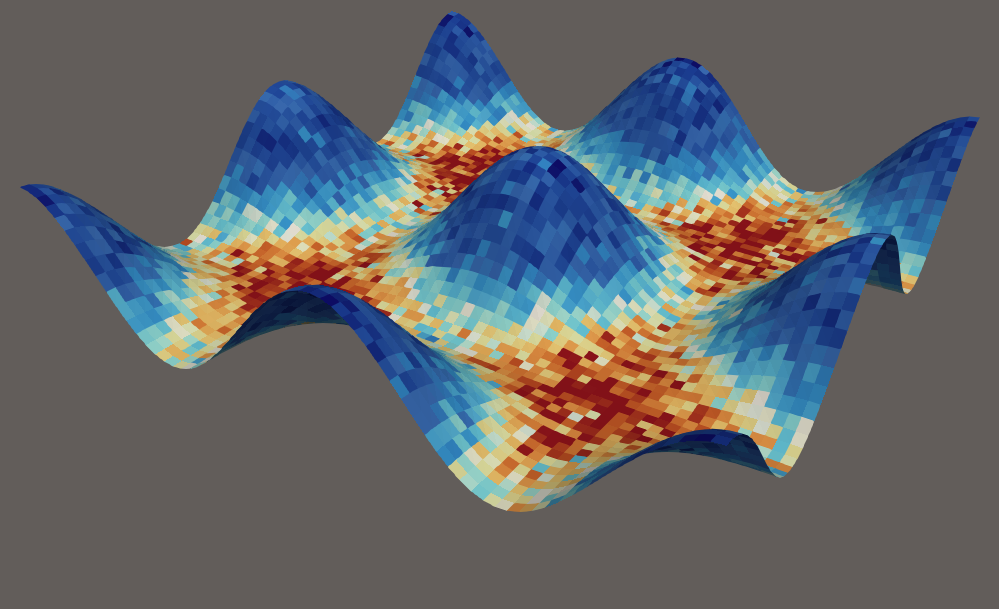}
 \includegraphics[width=0.4\textwidth]{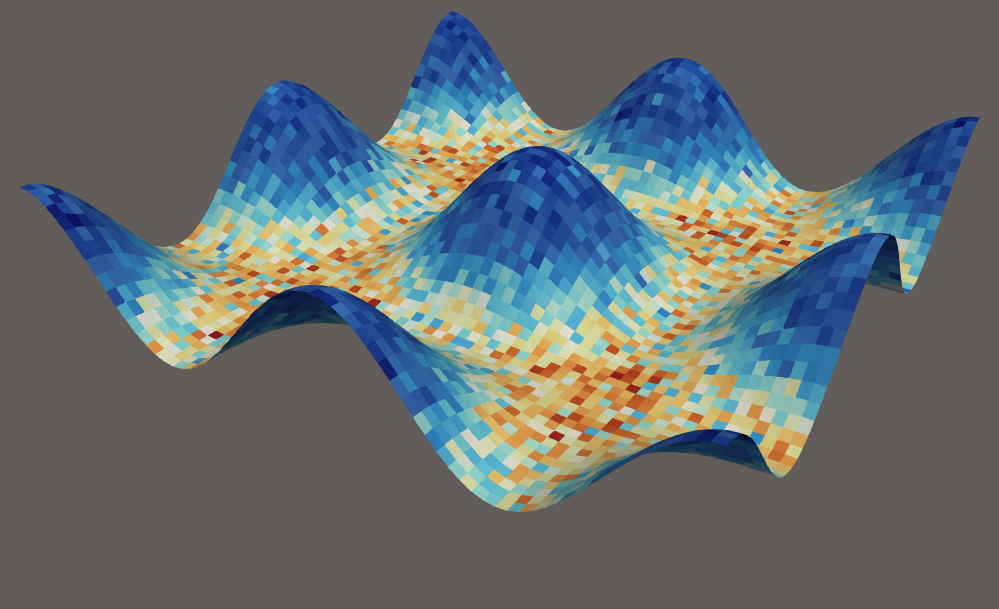}
    \caption{Time evolution of $\rho$. 
    The range of $\rho$ in the image is restricted to $[0,0.03]$. The height of the peaks on the surface have been scaled by 0.3 so that $\rho$ is not obscured.
    }
    \label{fig:phi_move_evolve}
\end{figure}

\begin{figure}[h]
    \centering
    t = 0.00 \hspace{2in} t = 0.15 \\
    \vspace{.05in}
    \includegraphics[width=0.4\textwidth]{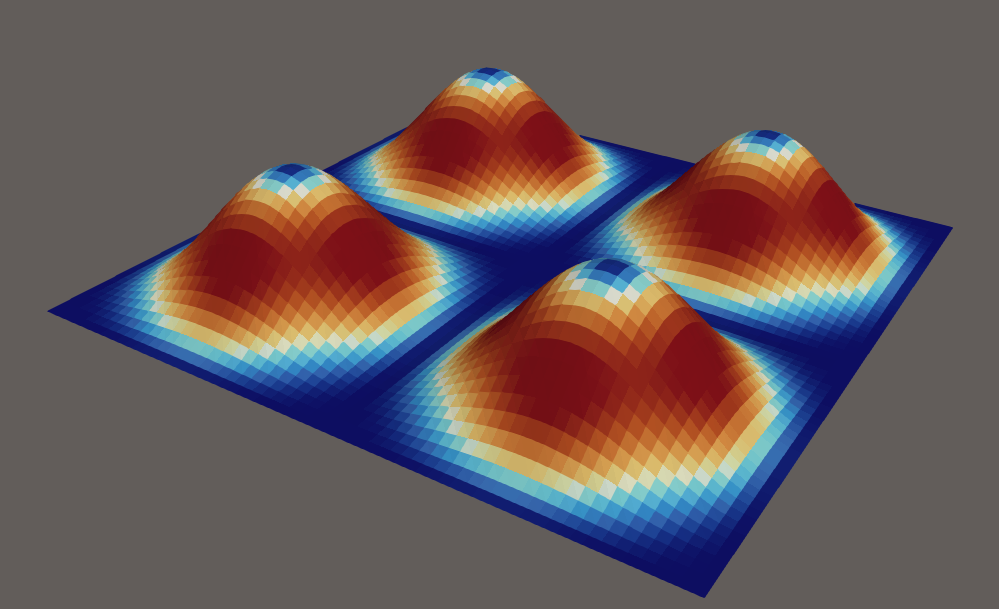}
    \includegraphics[width=0.4\textwidth]{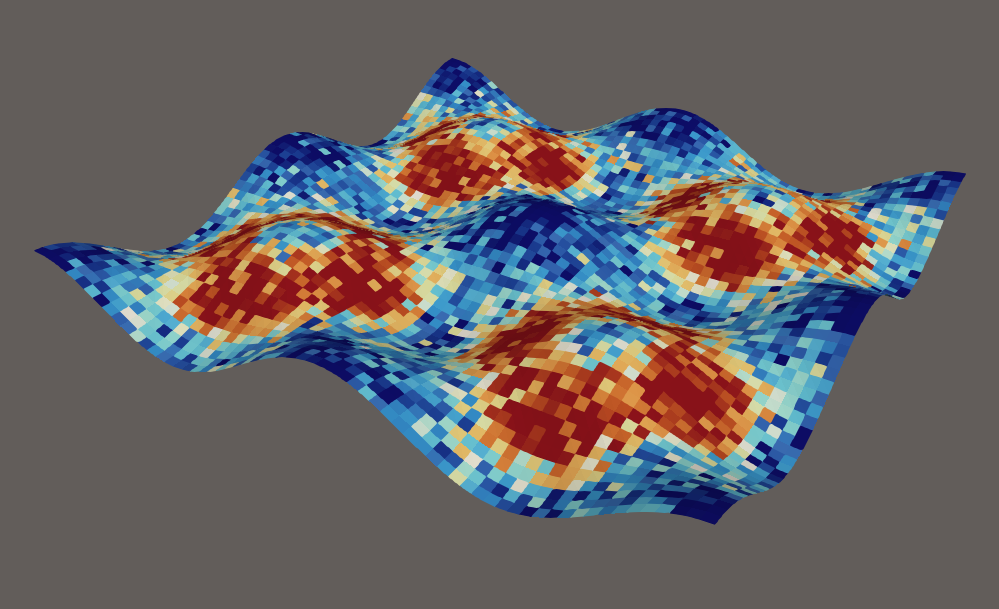} 
    
    t = 0.30 \hspace{2in} t = 0.56 \\
      \vspace{.05in}
 \includegraphics[width=0.4\textwidth]{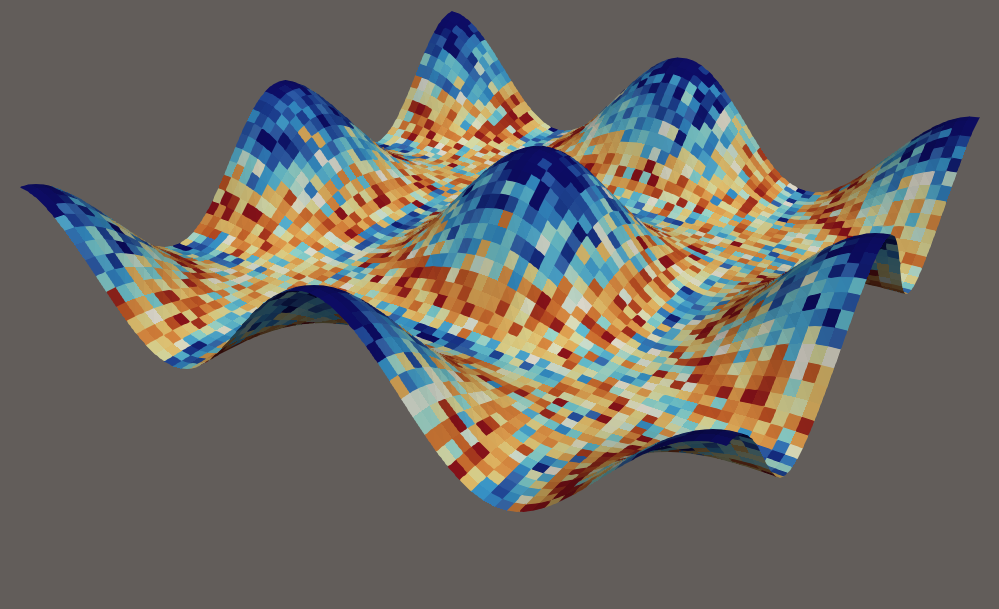}
 \includegraphics[width=0.4\textwidth]{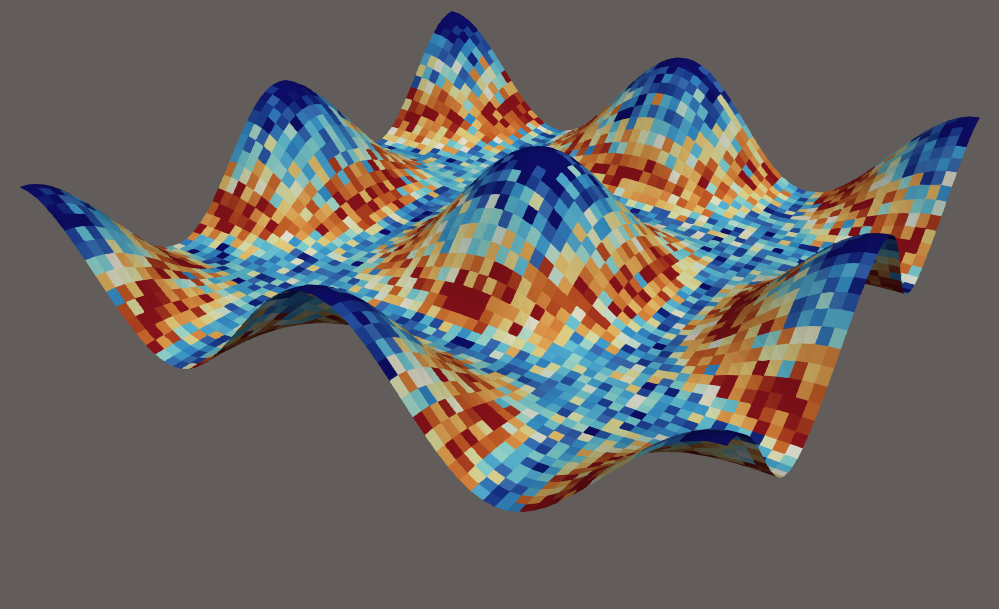}
    \caption{Time evolution of $N_{i,j}$. 
    The range in the image is restricted to $[10,40]$. The height of the peaks on the surface have been scaled by 0.3 so that $N_{i,j}$ is not obscured.
    }
    \label{fig:num_move_evolve}
\end{figure}

\appendix
\renewcommand{\thesection}{\appendixname~\Alph{section}}

\section{ Sinusoidal perturbation of a square}\label{AppendixA}

In order to avoid leaving the patch problem, we consider perturbed square with periodic identification: $ (x,y) \in [0,2\pi)^2 $.  The sinusoidal perturbation of a square is given by
\[
H(x,y) =a \sin(x) \sin(y), 
\]
where $a$ is some positive constant.

We can compute all the geometric quantities explicitly: First note that
\begin{align*}
    \nabla H(x,y) &= (a \cos x \sin y, a \sin x \cos x), \\
    G(x,y)
     &= \begin{pmatrix}
   1 + a^2 \cos^2 x \sin^2 y & a^2 \cos x \sin x \sin y \cos y\\[1ex]
   a^2 \cos x \sin x \sin y \cos y & 1 + a^2 \sin^2 x \cos^2 y
   \end{pmatrix}\\
    &= \begin{pmatrix}
   1+p^2 & pq \\[1ex]
   pq & 1+q^2
   \end{pmatrix},
\end{align*}
where 
$p = a \cos x \sin y$ and $q=a\sin x \cos y$.
We then have
\[
|G(x,y)| = 1+p^2+q^2 =: s(x,y),
\]
and
\[
G^{-1} = 
\begin{pmatrix}
  \dfrac{1+q^2}{s} & -\dfrac{pq}{s}\\[2ex]
   -\dfrac{pq}{s} & \dfrac{1+p^2}{s}
   \end{pmatrix}
   \qquad \mathrm{and} \qquad G^{-1/2} = \begin{pmatrix}
   1-c p^2 & -\,c pq\\[4pt]
-\,c pq & 1-c q^2
   \end{pmatrix},
\]
where 
\[
c = \frac{1}{s + \sqrt{s}}.
\]
We note that $G^{-1/2}$ is not unique; however, as long as $G^{-1/2}(G^{-1/2})^T = G^{-1}$, the noise term will have the correct covariance.

Finally, for particle simulations we need to evaluate
\[
b\begin{pmatrix}
      x \\ y 
   \end{pmatrix} = \frac{1}{\sqrt{| G |}} \nabla \cdummy \left(
   \sqrt{| G |} G^{- 1} \right) \begin{pmatrix}
      x \\ y 
   \end{pmatrix}
\]
in Equation (\ref{eq:BMonSurface}):
\begin{align}
b^x
 &= \frac{a^2\big(2+a^2(\cos^2x+\cos^2y)\big)\; \sin x \, \cos x \,\sin^2 y}
        {s^2}, \\[2ex]
b^y
 &= \frac{a^2\big(2+a^2(\cos^2x+\cos^2y)\big)\; \sin y \, \cos y \,\sin^2 x}
        {s^2} .
\end{align}

\section{Four peak surface}\label{AppendixB}

The second surface we consider features four peaks on the domain
$(x,y) \in [0,2\pi)^2 $.  The surface in this case is given by
\[
H(x,y) = a \sin^2(x) \sin^2(y)
\]
where $a$ is some positive constant.

All the geometric quantities can be compute explicitly as:
\begin{align*}
    \nabla H(x,y) &= \begin{pmatrix}
        2 a \cos x \, \sin x \,\sin^2 y  \\ 2a \sin^2 x \,\cos y \,\sin y
    \end{pmatrix}, \\
    G(x,y)
     &= \begin{pmatrix}
   1 + 4 a^2 \cos^2 x \,\sin^2 x \,\sin^4 y & 4 a^2 \cos x\, \sin^3 x\, \sin^3 y \,\cos y\\[1ex]
   4 a^2 \cos x \,\sin^3 x \,\sin^3 y\, \cos y  & 1 + 4 a^2 \sin^4 x \,\cos^2y \,\sin^2y
   \end{pmatrix}.\\
    &= \begin{pmatrix}
   1+p^2 & pq \\[1ex]
   pq & 1+q^2
   \end{pmatrix},
   \end{align*}
where $p =  2 a \cos x  \sin x \sin^2 y$ and $q = 2a \sin^2 x \cos y \sin y$.
We then have
\[
|G(x,y)| = 1+p^2+q^2 =: s(x,y),
\]
and
\[
G^{-1} = 
\begin{pmatrix}
  \dfrac{1+q^2}{s} & -\dfrac{pq}{s}\\[2ex]
   -\dfrac{pq}{s} & \dfrac{1+p^2}{s}
   \end{pmatrix}
   \qquad \mathrm{and} \qquad G^{-1/2} = \begin{pmatrix}
   1-c p^2 & -\,c pq\\[4pt]
-\,c pq & 1-c q^2
   \end{pmatrix},
\]
where 
\[
c = \frac{1}{s + \sqrt{s}}.
\]
We again note that $G^{-1/2}$ is not unique; however, as long as $G^{-1/2}(G^{-1/2})^T = G^{-1}$, the noise term will have the correct covariance.

Finally, for particle simulations we need to evaluate
\[
b\begin{pmatrix}
      x \\ y 
   \end{pmatrix} = \frac{1}{\sqrt{| G |}} \nabla \cdummy \left(
   \sqrt{| G |} G^{- 1} \right) \begin{pmatrix}
      x \\ y 
   \end{pmatrix}
\]
in Equation~(\ref{eq:BMonSurface}).
We first compute the derivatives of $s$ with respect to $x$ and $y$ to obtain
\begin{eqnarray*}
    s_x =& 8 a^2 \sin x \,\sin y \,\cos x \left( 2 \cos^2y\, \sin^2x \,\sin y +
    \sin^3 y\, (\cos^2 x -\sin^2x) \right) \\
    s_y =& 8 a^2 \sin x \,\sin y \,\cos y \left( 2 \cos^2x \,\sin^2y\,\sin x +
    \sin^3 x\, (\cos^2 y -\sin^2y) \right).
\end{eqnarray*}
We then have
\begin{eqnarray*}
    b^x =& 2 a^2 \left( \frac{2 \cos x \,\sin^3 x\, \sin^2y}{s} +
    \frac{[\cos x \, \cos y \, \sin^3 x \, \sin^3y \,] s_y}{s^2}  \right)
    -\frac{[4a^2 \cos^2y \,\sin^3x\,\sin^3y\,+ 1] s_x}{2 s^2}\\
    b^y =&  2 a^2 \left( \frac{2 \cos y \,\sin^3 y\, \sin^2x}{s} +
    \frac{[\cos x \, \cos y \, \sin^3 x \, \sin^3y \,] s_x}{s^2}  \right)
    -\frac{[4a^2 \cos^2x \,\sin^3x\,\sin^3y\,+ 1] s_y}{2 s^2} .
\end{eqnarray*}

\section{Additional proofs}\label{app:aux-lemma}

\begin{proof}[Proof of Lemma \ref{lem:martingale-filtration}]
  By a monotone class argument, it suffices to show that
  \[ \mathbb{E} [(M_t - M_s) F_s \Phi (B)] = 0 \]
  for all $0 \leq s \leq t$ and all bounded and
  $\mathcal{F}_s$-measurable $F_s$ and all bounded measurable $\Phi : C
  (\mathbb{R}_+, \mathbb{R}^d) \rightarrow \mathbb{R}$. By the martingale
  representation theorem for $B$, there exists a square-integrable adapted $H$ with
  \[
    \Phi (B) =\mathbb{E} [\Phi (B)] +
  \int_0^{\infty} H_r \cdot \mathd B_r,
  \]
  hence
  \begin{align*}
    \mathbb{E} [(M_t - M_s) F_s \Phi (B)] & = \mathbb{E} \left[ (M_t - M_s)
    F_s \left( \mathbb{E} [\Phi (B)] + \int_0^s H_r \cdot \mathd B_r \right)
    \right] \\
    &\qquad +\mathbb{E} \left[ (M_t - M_s) F_s \int_s^{\infty} H_r \cdot
    \mathd B_r \right]\\
    & = 0,
  \end{align*}
  where we used the tower property to insert $\mathbb{E} [\cdot
  \mid\mathcal{F}_s]$ for both terms, and for the second term we also used that the product of $M$ with the stochastic integral is a martingale, because
  \begin{align*}
    (M_{\cdot \wedge t} - M_s) \int_s^{\cdot} H_r \cdot \mathd B_r & =
    (M_{\cdot \wedge t} - M_s) \int_s^{\cdot} H_r \cdot \mathd B_r -
    \int_s^{\cdot \wedge t} H_r \cdot \mathd \langle M, B \rangle_r\\
    & = (M_{\cdot \wedge t} - M_s) \int_s^{\cdot} H_r \cdot \mathd
    B_r - \left\langle (M_{\cdot \wedge t} - M_s), \int_s^{\cdot} H_r
    \cdot \mathd B_r \right\rangle.
  \end{align*}
\end{proof}

\section*{Acknowledgment}
The authors are grateful to Ann Almgren for her help with numerical experiments and Mirjana Djori\' c for informative discussions about the geometric setting.
ADj and NP gratefully acknowledge funding   by Deutsche Forschungsgemeinschaft (DFG)  through grant CRC 1114 ``Scaling Cascades in Complex Systems", Project Number 235221301, Project C10 ``Numerical
Analysis for nonlinear SPDE models of particle systems". This material is based upon work supported by the National Science Foundation under Grant No. DMS-2424139, while ADj and NP were in residence at the Simons Laufer Mathematical Sciences Institute in Berkeley, California, during the Fall 2025 semester. Further support of ADj is by Daimler and Benz Foundation as part of the scholarship program for junior professors and postdoctoral researchers.
The work of JB was supported by the U.S. Department of Energy, Office of Science, Office of Advanced Scientific Computing Research, Applied Mathematics Program under contract No. DE-AC02-05CH11231. 

\bibliographystyle{plain}   
\bibliography{ref} 
\end{document}